\documentclass[11pt]{amsart}
\usepackage{graphicx}
\usepackage{amssymb}
\usepackage{epstopdf} %needed for pdftex
\usepackage{pdfsync}

\usepackage{ifthen}
\usepackage{amsmath}
\usepackage{color}
\usepackage{hyperref} %hyperref is incompatible (now ok) with arXiv

\usepackage[all]{xy}
\xyoption{matrix}
\xyoption{arrow}

\newtheorem{thm}{Theorem}[subsection]
\newtheorem{lem}[thm]{Lemma}
\newtheorem{cor}[thm]{Corollary}
\newtheorem{prop}[thm]{Proposition}

\theoremstyle{definition}
\newtheorem{defn}[thm]{Definition}

\theoremstyle{remark}
\newtheorem{rem}[thm]{Remark}

\numberwithin{equation}{section}

\def\vs#1{\vskip .#1 cm} %enter amount of skip wanted at #1
\def\noi{\noindent}

\def\xrarrow{\xrightarrow} %right arrow {label on top}

\def\into{\hookrightarrow}%or \rightarrowtail

\def\rel{{\rm\ rel\ }}

\def\smallcoprod{\,{\textstyle{\coprod}}\,}

\def\<{\left<}
\def\>{\right>}

\DeclareMathOperator{\interior}{int}%
\DeclareMathOperator{\colim}{colim}%

\newcommand{\Gamsub}[2]{\Gam_{#1,#2}}

\newcommand{\ttd}[1]{\widetilde\cS^{t/d}_{#1}} % unstable tangential smooth structures
\newcommand{\std}[2]{\widetilde\cS^s_{#1,#2}} %stable tangential smooth structures, relative case
\newcommand{\stdone}[1]{\widetilde\cS^{s}_{#1}}%stable tangential smooth structures, with just B.

\newcommand{\IK}{\text{\rm IK}}

\newcommand{\field}[1]{\mathbb{#1}}
\newcommand{\ZZ}{\ensuremath{{\field{Z}}}}

\newcommand{\RR}{\ensuremath{{\field{R}}}}
\newcommand{\QQ}{\ensuremath{{\field{Q}}}}

\newcommand{\commentout}[1]{}

\newcommand{\size}[1]{\ensuremath{\vert #1 \vert}}

\newcommand{\vertical}{{\text{\sf v}}}

\def\vv{^\vertical\!}
\def\vvv{^\vertical}

\def\ll{\lambda}

\newcommand{\cC}{\ensuremath{{\mathcal{C}}}}

\newcommand{\cH}{\ensuremath{{\mathcal{H}}}}

\newcommand{\cS}{\ensuremath{{\mathcal{S}}}}

\def\a{\alpha}
\def\b{\beta}
\def\g{\gamma}
\def\Gam{\Gamma}
\def\d{\partial}

\def\e{\epsilon}
\def\f{\varphi}

\def\r{\rho}
\def\s{\sigma}
\def\Sig{\Sigma}
\def\t{\tau}
\def\th{\theta}

\def\w{\omega}

\def\ov{\overline}

\def\st{\,|\,}

\newcommand{\subsubsub}[2]{\vs2\noi {#1}. \emph{#2}}

\title[Exotic smooth structures I]{Exotic smooth structures on topological fibre bundles I}

\author{Sebastian Goette}
\address{Mathematisches Institut, Universit\"at Freiburg, Eckerstr. 1, 79104 Freiburg, Germany}
\email{sebastian.goette@math.uni-freiburg.de}

\author{Kiyoshi Igusa}
\address{Department of Mathematics, Brandeis University, Waltham, MA 02454}
\email{igusa@brandeis.edu}

\author{Bruce Williams}
\address{Department of Mathematics, University of Notre Dame, Notre Dame, IN 46556}
\email{williams.4@nd.edu}


\begin{document}

\begin{abstract} When two smooth manifold bundles over the same base are fiberwise tangentially homeomorphic, the difference is measured by a homology class in the total space of the bundle. We call this the relative smooth structure class. Rationally and stably, this is a complete invariant. We give a more or less complete and self-contained exposition of this theory which is a reformulation of some of the results of \cite{DWW}.

An important application is the computation of the Igusa-Klein higher Reidemeister torsion invariants of these exotic smooth structures. Namely, the higher torsion invariant is equal to the Poincar\'e dual of the image of the smooth structure class in the homology of the base. This is proved in the companion paper \cite{First} written by the first two authors. 
\end{abstract}

\maketitle

\tableofcontents

%\setcounter{section}{-1}% one less than the section number
%\setcounter{page}{1}% equal to the page number
 
 %%%%%%%%%%%%%%%%%%%%%%%%%%%%%%%%%%%%
 %
 %						Main Body
 %
 %%%%%%%%%%%%%%%%%%%%%%%%%%%%%%%%%%%%
%%\newpage

%	Part C: Arc de Triomph
%
%	today: % Wednesday, Nov 17, 2010: Starting again

 %%%%%%%%%%%%%%%%%%%%%%%%%%%%%%%%%%%
 %
 %		section 0 {Introduction and outline}
 %
 %%%%%%%%%%%%%%%%%%%%%%%%%%%%%%%%%%%

%\newpage

\section*{Introduction and outline}

%The main theorem of this paper is the following calculation with definitions given below. Suppose that $B$ is a compact oriented smooth ($C^\infty$) $q$-manifold whose boundary is a union of two $q-1$ manifolds $\d B=\d_0B\cup\d_1B$ with common boundary $\d\d_0B=\d\d_1B$. Suppose that $p:M\to B$ is a smooth fiber bundle with fiber $X$ another compact oriented smooth manifold.

Higher Reidemeister torsion is a cohomology class in the base of a smooth manifold bundle which can be used to distinguish between different smooth structures on the same topological manifold bundle. In other words, if $M\to B$ and $M'\to B$ are two smooth manifold bundles over the same base which are equivalent as topological bundles but have different higher Reidemeister torsion invariants then they cannot be equivalent as smooth bundles. The higher torsion is not always defined. However, given a fiberwise homeomorphism $f:M\to M'$ over $B$ one can always define the relative torsion $\t(M',M)\in H^\ast(B;\RR)$ which must vanish in order for $f$ to be fiberwise homotopic to a diffeomorphism.

There are three different definitions of the higher torsion due to Igusa-Klein \cite {IK1:Borel2}, \cite{I:BookOne}, Dwyer-Weiss-Williams \cite{DWW} and Bismut-Lott \cite {Bismut-Lott95}, \cite{BG2} which are now known to be related in a precise way \cite{BDKW}, \cite{Goette08}, \cite{I:Axioms0}. However, none of these is a complete invariant. There are differences in smooth structure which are not detected by the higher relative torsion. %The purpose of this paper is to show that 

The purpose of this paper is to reformulate the Dwyer-Weiss-Williams result \cite{DWW} which calculates, as a fiberwise generalized homology theory, the space of all stable fiberwise tangential smoothings of a compact topological manifold bundle and gives a complete rational invariant for this problem. In our reformulation, smooth structures on a topological manifold bundle are classified (rationally and stably) by a homology class in the total space of the bundle called the ``{smooth structure class}''. In the companion paper \cite{First}, the first two authors show that the higher relative \IK-torsion is the Poincar\'e dual of the image of the smooth structure class in the homology of the base.

We strive to give a complete description of classical results needed to derive the final statements in a useful format, especially for the application in our other paper. This paper is based on a handwritten outline written by the third author \cite{WilliamsNotes06} to explain this version of \cite{DWW} to the first two authors.

\subsection{Main results}

Let $p:M\to B$ be a smooth manifold bundle, i.e., a fiber bundle in which the fiber $X$, base $B$ and total space $M$ are compact smooth oriented manifolds. Suppose that $X$ is $N$-dimensional. Then the \emph{vertical tangent bundle} $T\vv M$ is the $N$-dimensional subbundle of the tangent bundle $TM$ of $M$ which is the kernel of the vector bundle morphism $Tp:TM\to TB$. 

An \emph{exotic smooth structure} on $M$ is another smooth manifold bundle $M'\to B$ with fiber $X'$ which is fiberwise \emph{tangentially homeomorphic} to $M$ in the sense that there is a homeomorphism $f:M\cong M'$ which commutes with projection to $B$ and which is covered by an isomorphism of vertical tangent bundles so that this linear isomorphism is compatible with the topological tangent bundle map associated to $f$. See Section \ref{subsec14} for details.

We need to \emph{stabilize}, by which we mean take the direct limit with respect to all linear disk bundles over the total space $M$. Then the space of stable fiberwise tangential smoothings of the bundle $M$ which we denote $\stdone{B}(M)$ is an infinite loop space and $\pi_0\stdone{B}(M)$ is a finitely generated abelian group. The main theorem (\ref{computation of stable smooth structures}) of this paper is:
\[
	\pi_0\stdone B(M)\otimes\RR\cong\bigoplus_{k>0} H_{\dim B-4k}(M,M_{\d B};\RR)\cong_D \bigoplus_{k>0} H^{\dim X+4k}(M,\d\vv M;\RR)
\]
where $M_{\d B}=p^{-1}(\d B)$ and $\d\vv M$ is the vertical boundary of $M$, which is the fiber bundle over $B$ with fiber $\d X$ (so that $\d M=M_{\d B}\cup \d\vv M$) and $D$ is Poincar\'e duality.
In fact we consider the relative version $\std{B}{\d_0B}(M)$, which is the group of all stable exotic smooth structures on $M$ which agree with the given smooth structure over $\d_0B\subset \d B$.

Any exotic smooth structure $M'$ on $M$ gives an element of $\stdone B(M)$. The image of this element in $\bigoplus_{k>0} H_{\dim B-4k}(M,M_{\d B};\RR)$ is called the \emph{smooth structure class} of $M'$ (relative to $M$) and denoted $\Theta(M',M)$. The main theorem of the companion paper \cite{First} is that the Poincar\'e dual of the relative higher \IK-torsion invariant is the image of the smooth structure class in $H_\ast(B,\d B)$:
\begin{equation}\label{tIK is the pushdown of Theta}
	D\t^\IK(M',M)=p_\ast \Theta(M',M)\in \bigoplus_{k>0} H_{\dim B-4k}(B,\d B;\RR)
\end{equation}
In the case when $M\to B$ is a linear disk bundle, this is a reformulation of the well-known result of Farrell-Hsiang \cite{FarrellHsiangDiffDn}. We prove this in Proposition \ref{prop: normalization of theta}. 

In order to obtain \eqref{tIK is the pushdown of Theta} in general we need a construction to embed the disk models into any smooth bundle and another theorem which relates the smooth structure of a subbundle with that of the whole bundle. In the case of embedded subbundles, we can simply use the naturality of the smooth structure map $\Theta$. Unfortunately, there are not enough embedded subbundles, so we need stratified subbundles (Definition \ref{def:stratified embedding-immersion pair}) to generate all of the homology of $M$ (in the right degrees). The exotic smooth structures on $M$ generated by these stratified subbundles play an essential role in our main application \cite{First}. So, we need an extension of the naturality of $\Theta$ to the stratified case. This is Corollary \ref{cor of stratified homotopy}.

% subsection

%%\newpage
%-----------------------------------------------------------------------------------
%            sub section {Basic definitions}
%-----------------------------------------------------------------------------------

%In Section \ref{Appendix A} we explain Dwyer-Weiss-Williams smoothing theory. 

\subsection{Outline of paper}

This paper has two sections. In the first section we construct the space $\stdone B(M)$ of stable tangential smoothings of a compact manifold bundle and show that it is homotopy equivalent to a space of sections of a certain fiberwise homology bundle and we show that this homotopy equivalence is compatible with stratified embeddings.

Section \ref{Appendix A} is divided into five parts:

(\ref{subsec11}) \emph{Spaces of manifolds.} We define spaces of manifolds as simplicial sets. Then a smooth/topological bundle over $B$ is a continuous mapping into the realization of the corresponding simplicial set. The smoothing problem for bundles over $B$ is the problem of computing the homotopy fiber of the mapping $|\cS_\bullet^d(n)|^B\to |\cS_\bullet^t(n)|^B$ induced by the simplicial forgetful map $\f:\cS_\bullet^d(n)\to \cS_\bullet^t(n)$ from the space of smooth manifolds to the space of topological manifolds.

(\ref{subsec12}) \emph{The functor $EM$.} Euclidean bundles are topological bundles with fiber $\RR^n$ together with a zero section. A microbundle is the germ of such a bundle around the zero section. We define a microbundle to be the entire Euclidean bundle with morphisms of microbundles being germs at the zero section of maps. This gives an equivalent category with more convenient objects. We explain the well known fact that every paracompact topological manifold $M$ has a tangent Euclidean bundle $EM$ given as a neighborhood of the diagonal in $M\times M$. Then we construct the fiberwise tangent Euclidean bundle $E\vv M$ for any topological bundle $M\to B$.

(\ref{subsec13}) \emph{Linearized Euclidean bundle $VM$.} We define a \emph{linearized  manifold} to be a triple $(M,VM,\ll)$ where $M$ is a topological manifold $VM$ is a vector bundle over $M$ and $\ll:VM\to EM$ is an equivalence of microbundles. We define a \emph{tangential homeomorphism} between linearized manifolds and show that such a structure is equivalent to a one parameter family of linearized manifolds going from the domain to the target. A smooth manifold is linearized by its tangent bundle and an exponential map $TM\to EM$. We refer to a tangential homeomorphism as ``smooth'' if it is a diffeomorphism.

%By classical smoothing theory, this additional structure exists if and only if $M$ is smoothable (if $\dim M$ is large enough). If $M$ is smooth, then we can take $VM$ to be the usual tangent bundle $TM$.

(\ref{subsec14}) \emph{Tangential smoothing.} We define a \emph{tangential smoothing} of a linearized manifold $(M_0,VM_0)$ to be a tangential homeomorphism with a smooth manifold $(M_1,TM_1)$. Proposition \ref{the homotopy fiber of the map Sd to lin St} shows that the space of tangential smoothings of a fixed linearized manifold is the homotopy fiber of the simplicial forgetful map 
$
	\widetilde\f:\cS_\bullet^d(n)\to \widetilde\cS_\bullet^t(n)
$ where $ \widetilde\cS_\bullet^t(n)$ is the simplicial set of linearized topological $n$ manifolds.

Let $\ttd{B,A}(M)$ be the space (simplicial set) of all fiberwise tangential smoothings (relative to $A$) of a linearized topological bundle $M\to B$ (Definition \ref{def: fiberwise tangential smoothings}). Then we define the \emph{stable tangential smoothing space} to be the direct limit
\[
	\std{B}{A}(M_0):=\lim_\to
	\ttd{B,A}(D(\xi))
\]
where the direct limit is with respect to all linear disk bundles $\xi$ over $M_0$. This definition has the effect of ``forgetting'' the tangent bundle of $M_0$.

(\ref{subsec15}) \emph{Smoothing theorems.} In this section we prove the main smoothing theorem \ref{main smoothing thm} which gives a homotopy equivalence
\[
	\g_W: \std{B}{\d_0B}(W)\simeq\Gamsub{B}{\d_0B}\cH^\%_B(W)
\]
if $\d B=\d_0B\cup\d_1B$ where $\cH^\%_B(W)$ is the fiberwise $\cH^\%$ homology bundle over $B$, i.e., the bundle whose fiber over $b$ is $\cH^\%(W_b)$ where $\cH^\%$ is the homology theory associated to the $\Omega$-spectrum of the infinite loop space
\[
	\cH(\ast)=\stdone\ast(\ast)
\]
which is the space of all stable tangential smoothings of a point $\ast$ considered as a bundle over itself.

The computation of $\std{B}{\d_0B}(W)$ is broken up into several subsections of Section \ref{subsec15}. Given a fixed smooth bundle $W\to B$ and another bundle $W'\to B$ which is fiberwise tangentially equivalent to $W$, the image of $W'$ in $\std{B}{\d_0B}(W)$ is the obstruction to finding a stable fiberwise diffeomorphism $W\to W'$. Using immersion theory we show that there is no obstruction to finding this fiberwise diffeomorphism on the core of $W$. (The \emph{core} being any codimension $0$ subbundle of the fiberwise interior of $W$ which is a fiberwise deformation retract.) Then, by a theorem of Morlet, an extension of this diffeomorphism to all of $W$ is the same as a section of the bundle over $B$ with fiber 
\[
\cH^\%(X)=\Omega^\infty(X_+\wedge\cH(\ast))
\]
which is a homology theory in the fiber $X$.

Finally, Section \ref{subsec15} ends with a discussion of stratified subbundles. A {stratified subbundle} $E^\delta$ of a smooth bundle $W\to B$ is the image $\tilde\ll(E)$ of a stratified embedding-immersion pair $(\tilde\ll,\ll)$ with a certain part deleted. A \emph{stratified embedding} from another smooth bundle $\pi:E\to L$ into $p:W\to B$ is defined to be a pair consisting of a codimension 0 embedding $\tilde\ll:E\to W$ and a codimension 0 immersion $\ll:L\to B$ which has transverse self-intersections and meets the boundary of $B$ transversely so that $p\circ\tilde\ll=\ll\circ\pi$ (Definition \ref{def:stratified embedding-immersion pair}). Let $\d_0L$ be the closure of the complement in $\d L$ of $\ll^{-1}(\d B)$. Then $E^\delta$ is the image $\tilde\ll(E)$ with $\tilde\ll(E|{\d_0L})$ deleted. Corollary \ref {corollary of stratified smoothing theorem} proves the key result that the homotopy equivalence $\g_W$ is compatible with stratified embeddings.
\vs3

In Section \ref{Appendix B} we review elementary homotopy theory (Section \ref{subsec21}) and use it in Section \ref{subsec22} to calculate, rationally, $\pi_0$ of the spaces of sections obtained in Section \ref{Appendix A}. The main result is Corollary \ref{main homotopy calculation}: Suppose that the base $B$ and fiber $X$ of $M\to B$ are oriented manifolds. Then we have an isomorphism
\[
	\th_M:\pi_0\Gamsub{B}{\d_0B}\cH^\%_B(M)\otimes\RR\cong \bigoplus_{k>0} H_{\dim B-4k}(M,\d_1M;\RR)%\cong \bigoplus_{k>0} H^{N+4k}(M,\d_0M;\RR)
\]
where $\d_1M=M_{\d_1B}$. Together, with the homotopy equivalence $\g_M$ obtained in Section \ref{Appendix A}, this calculations give the isomorphism:
\[
	\Theta_M=\th_M\circ(\g_M)_\ast:\std{B}{\d_0B}(M)\otimes\RR\xrarrow\simeq
	\bigoplus_{k>0} H_{q-4k}(M,\d_1M;\RR)
\]
Corollary \ref{main homotopy calculation} follows from Theorem \ref{thm: homotopy type of space of sections} which is proved in Section \ref{subsec23}. Finally, in Section \ref{subsec24}, we show that these isomorphisms are compatible with stratified embeddings. In particular, we obtain the crucial fact (Corollary \ref{main corollary}), needed in our other paper, that $\Theta$ is compatible with stratified embedding.

%%%%%%%%%%%%%%%%%%%%%%%%%%%%%

\subsection{Acknowledgments} 

This paper originated at the 2006 Arbeitsgemeinshaft at Oberwolfach on ``Higher Torsion Invariants in Differential Topology and Algebraic K-Theory.''  The American Institute of Mathematics in Palo Alto also helped us to finish this project by hosting a workshop on higher Reidemeister torsion in 2009. Both meeting were very stimulating and productive and without them this work would never have been completed. Also, we would like to thank the referee for many suggestions which greatly improved this paper and the companion paper \cite{First}.

%%%%%%%%%%%%%%%%%%%%%%%%%%%%%%
%%%%%%%%%%%%%%%%%%%%%%%%%%%%%%
%
%			Appendix A {Smoothing theory}
%
%%%%%%%%%%%%%%%%%%%%%%%%%%%%%%
%%%%%%%%%%%%%%%%%%%%%%%%%%%%%%

%	Exotic smooth structures:  
%
%	This is part A, sections 0-4
%
%	0. Spaces of manifolds
%	1. The functor EM
%	2. Linearized Euclidean bundle VM
%	3. Tangential smoothing
%	4. Smoothing theorems
%

%\setcounter{section}{0}% one less than the section number
%\setcounter{page}{1}% equal to the page number

%\newpage

\section{Smoothing theory}\label{Appendix A}

%\centerline{\bf Appendix A. {Smoothing theory}}\vs4

%We want to establish the precise definition of the object of study of this paper, namely the set of smooth structures of a topological manifold bundle $M\to B$ for which a lifting of the tangent microbundle to the orthogonal group has already been chosen. We will also add orientation and stability conditions at the end.

%\setcounter{section}{-1}

%%%%%%%%%%%%%%%%%%%%%%%%%%%%%%
%
%				subSection  {Space of manifolds}
%
%%%%%%%%%%%%%%%%%%%%%%%%%%%%%%

\subsection{Spaces of manifolds}\label{subsec11}

First we recall the basic construction which is that a compact topological/smooth manifold bundle $M\to E\to B$ is equivalent to a mapping from $B$ to the space of topological/smooth manifolds which are homeomorphic/diffeomorphic to $M$. This leads us to consider two spaces of manifolds: topological manifolds and smooth manifolds and the homotopy fiber of the forgetful map
\[
	\f: \cS_\bullet^d(n)\to \cS_\bullet^t(n)
\]
which is the space of all smoothings of a fixed topological $n$-manifold. We also need one other space of manifolds: topological manifolds with linear structures on their tangent Euclidean bundles. We will discuss this after we recall the theory of Euclidean bundles in the next section.

%This smoothing problem is divided into two parts. First we need to reduce the structure group of the ``topological tangent bundle'' to $O(n)$. Then we need to find a smooth structure having a prescribed linear tangent bundle. This leads to a third space of manifolds which we call ``linearized'' topological manifolds. These are topological manifolds with linear (vector bundle) structures on their tangent Euclidean bundles. Euclidean bundles will be recalled in detail in the next section.

%We assume that all manifolds are topological submanifolds of $\RR^\infty$. A \emph{smooth structure} on $M$ will be given by a second (topological) embedding $M\into \RR^\infty$ whose image is a smooth submanifold of $\RR^\infty$. This will allow us to consider different smooth structures on the same topological manifold. We will also see that the space of smoothings of a topological manifold is homotopy equivalent to the space of ``smoothing'' of the manifold.

%%%\newpage
%-------------------------------------------------------------------------
%-------------------------------------------------------------------------
%			sub sub section {Topological manifolds}
%-------------------------------------------------------------------------
%-------------------------------------------------------------------------

\subsubsection{Topological manifolds} Let $S^t(n)$ be the set of all compact topological $n$-submanifolds of $\RR^\infty$ so that \[
(M,\d M)\subset ([0,\infty)\times \RR^\infty,0\times\RR^\infty)
\]
This condition allows us to attach a canonical external open collar $C=\d M\times(-1,0]\subset \RR^\infty\times(-1,0]$. Instead of defining a topology on the space $S^t(n)$ we will take the easy approach of making it into a simplicial set. So, let $\cS_\bullet^t(n)$ be the simplicial set whose $k$-simplices are continuous $\Delta^k$ families of compact topological $n$-manifolds $M_t,t\in \Delta^k$. By a \emph{continuous family} we mean a mapping $f:\Delta^k\to S^t(n)$ with the property that its adjoint
\[
	W=\{(t,x)\in\Delta^t\times \RR^\infty\st x\in M_t\}\subseteq \Delta^k\times\RR^\infty
\]
is a fiber bundle over $\Delta^k$ with fiber $f(t)=M_t$.

There is a \emph{tautological bundle} over the geometric realization $\size{\cS_\bullet^t(n)}$ whose fibers are compact $n$-manifolds embedded in $\RR^\infty$. The inverse image of a simplex is the corresponding manifold $W$ given above. Since this bundle contains all possible $W$, it is universal, i.e., any compact $n$-manifold bundle $p:E\to B$ over a triangulated space $B$ is classified by a mapping $B\to S^t(n)$ which is continuous on each simplex of $B$ in the above sense. Therefore, we get a simplicial map 
\[
	simp\ B\to {\cS_\bullet^t(n)}
\]
from the set of simplices of $B$ to ${\cS_\bullet^t(n)}$ which sends a $k$-simplex $\s$ to $p^{-1}(\s)\in \cS^t_k(n)$ (assuming that a fiberwise embedding $E\into B\times\RR^\infty$ has been chosen). This proves the following well-known theorem where $Homeo(M)$ is the group of homeomorphisms of $M$ with the compact-open topology.

\begin{thm} The geometric realization of $\cS_\bullet^t(n)$:
\[
	\size{\cS_\bullet^t(n)}\simeq\coprod BHomeo(M)
\]
is the disjoint union over all homeomorphism classes of compact $n$-manifolds $M$ of the classifying space of $Homeo(M)$.\qed
\end{thm}

%%%\newpage
%-------------------------------------------------------------------------
%-------------------------------------------------------------------------
%			subsub section {Smooth manifolds}
%-------------------------------------------------------------------------
%-------------------------------------------------------------------------

\subsubsection{Smooth manifolds} We define a smoothing of a topological manifold $M$ without boundary to be continuous mapping $\a:M\to \RR^\infty$ whose image is an immersed smooth submanifold. We call $\a$ an ``immersion''.

\begin{defn} If $M$ is a topological manifold and $N$ is a smooth manifold we define an \emph{immersion} to be a continuous mapping $\a:M\to N$ so that for every $x\in M$ there is an open neighborhood $U$ of $x$ in $M$ so that $\a$ is an embedding on $U$ with image $\a(U)$ a smooth submanifold of $N$ with smooth boundary. If $M_t$ is a family of topological manifolds without boundary forming the fibers of a bundle $W\to B$ over a smooth manifold $B$ then by a \emph{family of immersions} we mean a continuous family of maps $\a_t:M_t\to N$ giving an immersion $\a:W\to N\times B$.
\end{defn}

If $U$ is an open subset of the interior of a topological manifold $M$ we define a \emph{smoothing} of $U$ to be an immersion $\a:U\to \RR^\infty$. We defined a \emph{smoothing} of a closed subset $A$ of $M$ to be the germ of a smoothing of a neighborhood of $A$ in the interior of $M$. Thus a smoothing of $A$ is an equivalence class of smoothings in open neighborhoods of $A$ where two such smoothings are equivalent if they agree on a third smaller neighborhood of $A$. When we pass to manifold bundles we always define these open neighborhoods to be \emph{uniform open neighborhoods} meaning they are open subsets of the total space of the bundle. 

If $M$ is a manifold with boundary then we define a smoothing of $M$ to be the germ of a smoothing of a neighborhood of $M$ in $M'=M\cup C$ where $C=\d M\times(-1,0]$ is the standard external open collar for $M$. Smoothings of closed subsets of $M$ are defined similarly. 

The key point about this version of the definition of smoothing is that it is clearly \emph{excisive} in the following sense.

\begin{prop}\label{smoothing is excisive}
If $M$ is the union of two closed subsets $M=A\cup B$, then a smoothing $\a$ of $M$ is the same as a pair of smoothings $\a_A,\a_B$ for $A$ and $B$ which agree on $A\cap B$.
\end{prop}

\begin{proof}
The smoothings $\a_A,\a_B$ are given by immersions on neighborhoods of $A,B$ in $M'$ which agree on the intersection. This gives an immersion defined on a neighborhood of $M$ in $M'$.
\end{proof}

Let $S^d(n)$ be the set of all pairs $(M,\a)$ where $M$ is an element of $S^t(n)$, i.e. a compact topological submanifold of $\RR^\infty$ and $\a$ is a smoothing of $M$. Let $\cS_\bullet^d(n)$ denote the simplicial set whose $k$-simplices are pairs $(W,\a)$ where $W\in\cS_\bullet^t(n)$ and $\a:W\to \RR^\infty\times B$ is an immersion over $B$.

%compact smooth ($C^\infty$) manifold bundles over $\Delta^k$. To be very specific, we define $\cS^d_k(n)$ to be the set of all pairs $(W,\a)$ where $W\in \cS^t_k(n)$ is a compact topological manifold bundle over $\Delta^k$ and $\a:W\to \Delta^k\times\RR^\infty$ is a topological embedding over $\Delta^k$ whose image is a smooth submanifold of $\Delta^k\times\RR^\infty$ (with corners). Thus each fiber of the projection $\pi:W\to \Delta^k$ becomes a smooth submanifold of $\RR^\infty$ which we assume has no corners.

Note that there is a simplicial forgetful map
\[
	\f:\cS_\bullet^d(n)\to\cS_\bullet^t(n)
\]
which is given in every degree by projection to the first coordinate: $\f_k(W,\a)=W$.

\begin{thm} The geometric realization of $\cS_\bullet^d(n)$:
\[
	\size{\cS_\bullet^d(n)}\simeq\coprod BDif\!f(M)
\]
is weakly homotopy equivalent to the disjoint union over all diffeomorphism classes of compact $n$-manifolds $M$ of the classifying space of the group $Dif\!f(M)$ of diffeomorphisms of $M$ with the $C^\infty$ topology.\qed
\end{thm}

Our definition of smoothing also gives us an exponential map
\[
	\mu_M:TM\to M'
\]
defined in some neighborhood $U(M)$ of the zero section of the tangent bundle $TM\subset M\times \RR^\infty$. This is given by the inverse function theorem as the second coordinate of the inverse of $\pi:M\times M'\to TM$ given by $\pi(x,y)=\pi_x(\a'(y)-\a(x))$ where $\pi_x$ is orthogonal projection to the tangent plane to $\a(M)$ at $\a(x)$ translated to the origin.

%-----------------------------------------------------------------------------------
%            subsub section Homotopy fiber of $\f$
%-----------------------------------------------------------------------------------

\subsubsection{Homotopy fiber of $\f$}\label{subsec:homotopy fiber of phi}

Given a single topological manifold $X$, the space of smoothings $\a$ of $X$ is homeomorphic to the space of all pairs $(X,\a)$. This is the same as the inverse image of $X$ (and its degeneracies) under the simplicial forgetful map $\f:\cS_\bullet^d\to\cS_\bullet^t$. An elementary argument shows that this space is the homotopy fiber of the simplicial map $\f$. More generally, the space of smoothings of a continuous $\Delta^k$ family $W_0$ of topological manifolds is homotopy equivalent to the space of smoothings of $W_0$ which we define as follows.

\begin{defn} Suppose $W_0$ is a $\Delta^k$ family of topological $n$-manifolds, i.e. $(W_0\to \Delta^k)\in \cS^t_k(n)$. Then, a \emph{smoothing} of $W_0$ is defined to be a pair $(W,\a)$ where
\begin{enumerate}
\item $W$ is a continuous $\Delta^k\times I$ family of topological manifolds,
\item $W|\Delta^k\times0=W_0$,
\item $\a$ is a smoothing of $W_1=W|\Delta^k\times1$.
\end{enumerate}
\end{defn}

Since $W$ is topologically isomorphic to $W_0\times I$, the space of smoothings of $W_0$ is homotopy equivalent to the space of actual smoothings of $W_0$. Homotopy smoothings have some additional properties which we state without proof.

Let $\cS_\bullet^{t/d}(n)$ denote the simplicial set whose $k$-simplices are smoothings $(W,\a)$ of $k$-simplices $W_0\in\cS_k^t(n)$ as defined above. We have the following observation.

\begin{prop}\label{prop:smoothing = actual smoothing}
The simplicial forgetful map $\cS_\bullet^{t/d}(n)\to\cS_\bullet^d(n)$ sending $(W,\a)$ to $(W_1,\a)$ is a homotopy equivalence.\qed
\end{prop}

\begin{prop} If $X\in\cS_0^t(n)$ let
$\cS_\bullet^{t/d}(X)$ denote the simplicial subset of $\cS_\bullet^{t/d}(n)$ whose $k$-simplices are the smoothings of $X\times \Delta^k$. Then $|\cS_\bullet^{t/d}(X)|$ is the homotopy fiber of the forgetful map
\[
	|\f|:|\cS_\bullet^d(n)|\to|\cS_\bullet^t(n)|
\]
over $X\in \cS_0^t(n)$.\qed
\end{prop}

More generally, given any topological manifold bundle $W_0\to B$ where $B$ is a smooth manifold, we define a \emph{smoothing} of $W_0$ to be a topological embedding $\a:W_0\to B\times\RR^\infty$ over $B$ whose image is a smooth submanifold of $B\times \RR^\infty$. Smooth $\Delta^k$ families of such embeddings form a simplicial set which represents the space of all  homotopy liftings of the classifying map $B\to|\cS_\bullet^t(n)|$ to $|\cS_\bullet^d(n)|$, in other words a point in the homotopy fiber of the map
\[
	|\cS_\bullet^d(n)|^B\to |\cS_\bullet^t(n)|^B.
\]We call this the \emph{space of fiberwise smoothings} of $W_0$.

By a \emph{fiberwise smoothing} of $W_0$ we mean a pair $(W,\a)$ where $W$ is a topological manifold bundle over $B\times I$ which is equal to $W_0$ over $B\times 0$ and $\a$ is a smoothing of $W|B\times 1$. Taking $\Delta^k$ families we can form a simplicial set which we call the \emph{space of fiberwise smoothings} of $W_0$. As in Proposition \ref {prop:smoothing = actual smoothing}, it is clear that this space is homotopy equivalent to the space of actual smoothings of a topological bundle. However, this definition (developed by J. Lees \cite{Lees}) makes it easier to compare two different smoothings of the same bundle.  

Suppose that $f:W_0\to W_1$ is a homeomorphism between two smooth bundles over $B$. We can construct the associated smoothing
\[
	W=W_0\times I\cup_f W_1
\]
where $W_1$ is identified with $W_0\times I$ using the homeomorphism $f$. The problem is to determine whether there is a smooth structure on $W$ which extends the given smooth structure on $W_0$ and $W_1$. If $f$ is a smooth embedding then we can take the smooth structure on $W$ induced from $W_0\times I$. Using our strict definitions, this would be given by the family of smoothings $h_t=(1-t)\a_0+t\a_1\circ f:W_0\to\RR^\infty$. This will be a family of embeddings if we assume that the smoothings $\a_i$ of $W_i$ have image in linearly independent subspaces of $\RR^\infty$, if not we can simply pass through a third smoothing with this property.

\begin{prop}
Suppose that $f:W_0\to W_1$ is a fiberwise homeomorphism between smooth bundles over $B$ and suppose that $f$ is homotopic through continuous fiberwise embeddings over $B$ to a smooth embedding. Then there is a smooth structure on the fiberwise interior of $W$ which agrees with the smooth structure on $W_0$ and $W_1$.
\end{prop}

\begin{proof}
The continuous image of the fiberwise interior of $W_t$ under a topological embedding $f_t:W_t\to W_1$ is a smooth manifold. Therefore, the image of $f_t$ gives the desired smoothing.
\end{proof}

In classical smoothing theory, a smoothing of a topological manifold (of dimension $\ge5$) is given by a lifting of the tangent microbundle to a linear bundle. In other words, a vector bundle structure on the topological tangent bundle gives a smoothing of a single manifold. This is not true for a topological manifold bundle $W\to B$. If we choose a vector bundle structure on the vertical topological tangent bundle of the topological bundle $W\to B$ we have a further obstruction to smoothing. To study this obstruction we need to construct a third space of manifolds: compact topological manifolds with vector bundle structures on their topological tangent bundles. We call these ``linearized manifolds.''

%%%\newpage

%%%%%%%%%%%%%%%%%%%%%%%%%%%%%%
%
%				A: First part  {Tangential data}
%
%%%%%%%%%%%%%%%%%%%%%%%%%%%%%%

%\centerline{\bf Topological data}\vs6

%This part is divided into 3 parts:\vs3

%\newpage

%%%%%%%%%%%%%%%%%%%%%%%%%%%%%%%%
%
%				subSection  The functor EM
%
%%%%%%%%%%%%%%%%%%%%%%%%%%%%%%%%

\subsection{The functor {\it EM}}\label{subsec12}

A compact topological manifold $M$ has a topological tangent bundle 
\[
	\pi: EM\to M
\]
which is called the \emph{tangent Euclidean bundle} of $M$. The fiber $\pi^{-1}(x)$ is homeomorphic to an open ball neighborhood of $x$ in $M$. This section gives the basic properties of Euclidean bundles in general and the functorial properties of the tangent Euclidean bundle of a manifold.

%%%\newpage

%-------------------------------------------------------------------------
%-------------------------------------------------------------------------
%			sub sub section {Euclidean bundles}
%-------------------------------------------------------------------------
%-------------------------------------------------------------------------

\subsubsection{Euclidean bundles}
A \emph{Euclidean bundle} is a fiber bundle
\[
	\pi:E\to B
\]
with fiber $\RR^n$ and structure group $\cH omeo(\RR^n,0)$, the group of homeomorphisms of $\RR^n$ fixing $0$ with the compact open topology. By a \emph{Euclidean subbundle} of $E$ we mean an open subset $E_0\subseteq E$ which includes the zero section and which is also a Euclidean bundle. A \emph{morphism} of Euclidean bundles $E_0\to E_1$ over $B$ is a fiberwise open embedding which preserves the zero section, i.e., an isomorphism of $E_0$ with a subbundle of $E_1$.

\begin{lem}[Kister \cite{Kister64}]\label{Kister's lemma} Let $E\to B$ be a Euclidean bundle over a finite cell complex $B$ and let $E_0$ be a Euclidean subbundle of $E$. Then $E_0$ is fiberwise isotopic to $E$ fixing a neighborhood of the zero section. I.e., there is a continuous one-parameter family of Euclidean morphisms $f_t:E_0\to E$ which are the identity in a neighborhood of the zero section so that $f_0$ is the inclusion map and $f_1$ is a homeomorphism.\qed
\end{lem}

\begin{rem}\label{rem:Euclidean subbundles are isomorphic to the whole}
%Any topological embedding $f:X\to Y$ is isotopic to a homeomorphism if and only if there is a subbundle $Z$ of $Y\times I\to I$ so that $Z\cap Y\times 0=X\times0$ and $Y\times1\subseteq Z$ (since $Z\to I$ is isomorphic to $X\times I\to I$). In particular, the conclusion of the 
The conclusion of this lemma can be rephrased as saying that there is a Euclidean subbundle of $E\times I\to B\times I$ which is equal to $E_0$ over $B\times0$ and equal to $E$ over $B\times 1$.
\end{rem}

%Now we assume that $B$ is the geometric realization of a locally finite, finite dimensional simplicial complex, e.g., a triangulated manifold.

\begin{prop}
If $E,E'$ are Euclidean bundles over $B$ with isomorphic subbundles $\f_0:E_0\cong E_0'$ then there is an isomorphism $\f:E\to E'$ which agrees with $\f_0$ in some neighborhood of the zero section.
\end{prop}

\begin{proof}
$E\cong E_0\cong E_0'\cong E'$ by the lemma.
\end{proof}

\begin{lem}\label{lem:extension of E0 from subcomplex}
Let $E\to B$ be a Euclidean bundle over a finite dimensional CW-complex $B$ and let $E_0$ be a subbundle of the restriction $E|A$ of $E$ to a subcomplex $A\subseteq B$. Then there is a Euclidean subbundle $E_0'$ of $E$ so that $E_0=E_0'|A$.
\end{lem}

\begin{proof}
Assuming that we have constructed $E_0'$ over $A\cup B^k$ we can extend $E_0'$ to $A\cup B^{k+1}$ one $(k+1)$-cell at a time using the lemma as rephrased in the remark.
\end{proof}

%%%\newpage
%--------------------------------------------------------------------------------------------
%--------------------------------------------------------------------------------------------
%			subsub section Microbundles
%--------------------------------------------------------------------------------------------
%--------------------------------------------------------------------------------------------

\subsubsection{Microbundles}

A \emph{(topological) microbundle} over $B$ is defined to be a space $E$ containing $B$ with inclusion map $s:B\to E$ and retraction $\pi:E\to B$ so that $B$ is covered by open subsets $U$ of $E$ for which $\pi(U)=U\cap B$ and so that $\pi|U:U\to U\cap B$ is a Euclidean bundle.% The inclusion map $s:B\to E$ is a section of $\pi$.

A \emph{morphism} of topological microbundles $E_0\to E_1$ over $B$ is defined to be the germ along $B$ of a fiberwise homeomorphism $f:E_0\to E_1$ which is the identity on $B$. Thus, $f$ is only defined in some neighborhood of $B$ and any two such maps are equivalent if they agree on some neighborhood of $B$ in $E_0$. 

\begin{thm}[Kister, Masur]\label{thm: Kister Masur}
Suppose that $B$ is (the realization of) a finite dimensional, locally finite simplicial complex and $\pi:E\to B$ is a microbundle. Then $E$ contains an open neighborhood $E_0$ of $s(B)$ so that $(E_0,s(B))$ is a Euclidean bundle over $B$. Furthermore, $E_0$ is unique up to isomorphism.
\end{thm}

\begin{rem}
This implies that any topological microbundle over $B$ is microbundle isomorphic to a Euclidean bundle which is uniquely determined up to isomorphism.
\end{rem}

\begin{proof}
We can refine the triangulation of $B$ so that each simplex is contained in one of the open sets $U$. This gives a Euclidean bundle $E_\s$ over each simplex $\s$. Since $B$ is locally finite and finite dimensional we can choose the $E_\s$ so that $E_\s\subseteq E_\t|\s$ for all $\s\subseteq\t$. Then a Euclidean bundle $E_n$ can be constructed over the $n$-skeleton of $B$ by induction on $n$: $E_0$ already exists. Given $E_n$, we can extend $E_{n}$ to each $n+1$ simplex using Lemma \ref{lem:extension of E0 from subcomplex}. This shows existence.

To prove the uniqueness, we take any two Euclidean bundles and use the above argument to construct a third Euclidean bundle which is a subbundle of both. By Kister's Lemma \ref{Kister's lemma}, all three Euclidean bundles are isomorphic.
\end{proof}

\begin{cor}
If $B$ is dominated by a finite dimensional locally finite simplicial complex (for example a paracompact topological manifold) then any microbundle over $B$ contains a Euclidean bundle neighborhood of its section.
\end{cor}

\begin{proof}
If $B$ is a retract of a nice space $X$ then any microbundle over $B$ pulls back to a microbundle over $X$ which contains a Euclidean bundle which restricts to a Euclidean bundle neighborhood of the section of the original microbundle.

Any paracompact $n$-manifold $M$ satisfies this condition since each component of $M$ is second countable and therefore can be properly embedded in $\RR^{2n+1}$. This is an exercise in Munkres. The Tietze extension theorem can be used to show that $M$ is a retract of a neighborhood which we can take to be simplicial.
\end{proof}

%%%\newpage

%-------------------------------------------------------------------------
%-------------------------------------------------------------------------
%			subsub section {Tangent Euclidean bundle}
%-------------------------------------------------------------------------
%-------------------------------------------------------------------------

\subsubsection{Tangent Euclidean bundle}

The discussion above implies the well-known fact that every paracompact topological manifold $M$ has a tangent Euclidean bundle: we first attach the standard external open collar $C=\d M\times(-1,0]$ and embed $M$ in the open manifold $M'=M\cup C$. The \emph{tangent microbundle} of $M$ is the equivalence class of $p_1:M\times M'\to M$ (projection to the first coordinate), together with the diagonal section $\Delta:M\to M\times M'$.

\begin{defn} If $M$ is a topological manifold with external open collar $C=\d M\times(-1,0]$ and $\pi:E\to M$ is a Euclidean bundle then a \emph{topological exponential map} is defined to be a continuous map
\[
	\eta: E\to M'=M\cup C
\]
so that
\begin{enumerate}
\item $\eta(s(x))=x$ for all $x\in M$ where $s:M\to E$ is the zero section,
\item $\eta$ maps each fiber $E_x=\pi^{-1}(x)$ homeomorphically onto an open neighborhood of $x$ in $M'$.
\end{enumerate}
We call $(E,\eta)=(E,s,\pi,\eta)$ a \emph{tangent Euclidean bundle} for $M$.\end{defn}

If $(E,\eta)$ is a tangent Euclidean bundle for $M$ then the germ of $E$ around $s(B)$ is the tangent microbundle of $M$. To see this note that the two mappings $\pi,\eta$ form an open embedding $(\pi,\eta):E\into M\times M'$ by invariance of domain. The image of $(\pi,\eta)$ determines $(E,\eta)$ up to isomorphism. We say that $(E_0,\eta_0)$ is \emph{contained in} $(E_1,\eta_1)$ if the image of $E_0$ in $M\times M'$ is contained in the image of $E_1$. Kister's theorem on the existence and uniqueness of Euclidean bundles can be stated as follows in the case of the tangent microbundle.

\begin{thm}
Any paracompact topological manifold $M$ has a tangent Euclidean bundle. Furthermore, for any two tangent Euclidean bundles, there exists a third Euclidean bundle which is contained in both.\qed
\end{thm}

In the case where $M$ is a smooth manifold, we note that $EM$ is also a smooth manifold although the smooth structure of $EM$ along $\d M$ depends on a choice of extensions $\a'$ of the smooth structure of $M$ to $M'$. Also we have a canonical exponential map $\mu_M:TM\to M'$ defined in a neighborhood $U(M)$ of the zero section and this gives us a diffeomorphism of bundles between $EM$ and $U(M)$.

%%%\newpage

%-------------------------------------------------------------------------
%		subsub-section   {vertical tangent Euclidean bundle}
%-------------------------------------------------------------------------

\subsubsection{Vertical tangent Euclidean bundle}

We are interested in the case when $p:M\to B$ is a bundle over a finite complex $B$ with fiber $X$ a compact topological manifold. In this case we first add the external collar $C=\d\vv M\times(-1,0]$ where $\d\vv M$ is the \emph{vertical boundary} of $M$, i.e. the bundle over $B$ with fiber $\d X$. Then $M'=M\cup C$ is an open manifold bundle over $B$ with fiber $X'=X\cup \d X\times(-1,0]$. The fiber product $M\oplus M'$ is a bundle over $B$ with fiber $X\times X'$ and the vertical tangent microbundle is the neighborhood germ of the fiberwise diagonal $\Delta M$ in $M\oplus M'$. 

Since $M$ is a compact and finite dimensional it is dominated by a finite complex. Therefore, any microbundle over $M$ contains a Euclidean bundle. In particular, there exists a Euclidean bundle $\pi: E\vv M\to M$ unique up to isomorphism and an open embedding 
\[
	(\pi,\eta): E\vv M\into M\oplus M'
\]
over $M$ sending the zero section to the fiberwise diagonal.

We call $(E\vv M,\eta)$ a \emph{vertical tangent Euclidean bundle} for $M\xrarrow{p} B$ and $\eta:E\vv M\to M'$ the \emph{fiberwise topological exponential map}. As before, $(E\vv M,\eta)$ is given up to isomorphism by the image of the embedding $(\pi,\eta):E\vv M\into M\oplus M'$.

%%%\newpage
%-------------------------------------------------------------------------
%				sub	sub-section   {topological derivative}
%-------------------------------------------------------------------------

\subsubsection{Topological derivative}

Any homeomorphism $f:M_0\to M_1$ induces a homeomorphism $f':M_0'\to M_1'$ by sending $(x,t)\in \d M_0\times [0,1)$ to $(f(x),t)\in\d M_1\times[0,1)$. This gives a map of tangent microbundles:
\[
\xy
\xymatrix{
M_0\times M_0'\ar[rr]^{f\times f'}\ar[d]_{p_1}&& M_1\times M_1'\ar[d]^{p_1}\\
M_0 \ar[rr]^f && M_1
}
\endxy
\]
This morphism of microbundles is the \emph{topological derivative} of $f$. (If $E\to B,E'\to B'$ are bundles or microbundles, then a morphism $E\to E'$ over a map $f:B\to B'$ is defined to be a morphism between $E$ and the pull-back $f^\ast E'$ over $B$.)

Choosing Euclidean subbundles of the microbundles, we can represent the topological derivative by an isomorphism of Euclidean bundles $Ef$.
If $\pi_0:EM_0\to M_0$ is a tangent Euclidean bundle for $M_0$ with embedding $(\pi_0,\eta_0): EM_0\into M_0\times M_0'$ then $\pi_1=f\circ \pi_0:EM_0\to M_1$ is evidently a tangent Euclidean bundle for $M_1$ with embedding $(\pi_1,\eta_1)= ({f\times f'})(\pi_0, \eta_0): EM_0\to M_1\times M_1'$. In other words, there is an isomorphism of Euclidean bundles $Ef:EM_0\to EM_1$ over $f$ as indicated in the following commuting diagrams.
\[
\xy
\xymatrix{
EM_0\ar[d]_{\pi_0}\ar[r]^{Ef} &EM_1\ar[d]^{\pi_1} 
&EM_0\ar[d]_{\eta_0}\ar[r]^{Ef} &EM_1\ar[d]^{\eta_1}\\
M_0 \ar[r]^f & M_1
&M_0' \ar[r]^{f'} & M_1'
}
\endxy
\]
We call $Ef$ the \emph{topological (Euclidean) derivative} of $f$.

Similarly, if $f:M_0\to M_1$ is a fiberwise homeomorphism between two topological manifold bundles over the same space $B$, we get an isomorphism of Euclidean bundles $E\vv f:E\vv M_0\to E\vv M_1$ over $f$ compatible with the fiberwise exponential maps as indicated in the following diagrams in which all arrows commute with the projection to $B$.
\[
\xy
\xymatrix{
E\vv M_0\ar[d]\ar[r]^{E\vv f} &E\vv M_1\ar[d]
& E\vv M_0\ar[d]_{\eta_0}\ar[r]^{E\vv f} &E\vv M_1\ar[d]^{\eta_1}
\\
M_0 \ar[r]^{f} & M_1
& M_0' \ar[r]^{f'} & M_1'
}
\endxy
\]
We call $E\vv f$ the \emph{fiberwise} or \emph{vertical topological (Euclidean) derivative} of $f$.

%\newpage

%%%%%%%%%%%%%%%%%%%%%%%%%%%%%%%%
%
%			subSection  {Linearized Euclidean bundle $VM$}
%
%%%%%%%%%%%%%%%%%%%%%%%%%%%%%%%%

\subsection{Linearized Euclidean bundle {\it VM}}\label{subsec13}

The first step to finding a smooth structure on a topological manifold is to impose a linear structure on the tangent Euclidean bundle. 

%%%\newpage
%--------------------------------------------------------------------------------------------
%--------------------------------------------------------------------------------------------
%			subsub section linearization
%--------------------------------------------------------------------------------------------
%--------------------------------------------------------------------------------------------

\subsubsection{Linearization}
We define a \emph{linearization} of a Euclidean bundle $\pi:E\to M$ to be the germ along the zero section of a vector bundle structure on $E$. Since this is a germ, it is a structure on the microbundle of $E$. A linearization makes the microbundle linear. A linearization of $E$ is the same as a lifting of the structure map of $E$ to $BGL(n,\RR)$:
\[
\xymatrix{
&BGL(n,\RR)\ar[d]\\
M \ar@{-->}[ur]\ar[r] & BHomeo(\RR^n,0)
}
\]

A linearization of $E$ can also be viewed as an equivalence class $[\ll]$ of a topological microbundle morphism $\ll:V\to E$ where $V$ is a vector bundle over $B$, two such morphisms $\ll:V\to E,\ll':V'\to E$ being equivalent if $\ll=\ll'\circ\psi$ for some linear isomorphism $\psi:V\to V'$. In particular, $V$ is well-defined up to isomorphism. We call $\psi$ the \emph{comparison map} for $\ll,\ll'$. If $\ll,\ll'$ are inequivalent linearizations of $E$, we also get a comparison map $\psi=(\ll')^{-1}\ll:V\to V'$ which is a nonlinear map germ between vector bundles.

%We will often take $V$ to be a vector bundle over a space $B'$ with a homeomorphism $f:B'\to B$ and $\ll:V\to E$ to be a bundle map germ over $f$. In that case, $\ll$ will be called a \emph{linearization of $B$ over $f$}.

A \emph{linearization} of a topological manifold $M$ is defined to be a linearization of its tangent microbundle. This is given by a microbundle morphism
\[
	\ll:VM\to EM
\]
for some vector bundle $VM$. We call the pair $(M,VM)$ a \emph{linearized manifold}. A \emph{fiberwise linearization} of a topological manifold bundle $W\to B$ is defined to be a linearization $V\vv W$ of the vertical tangent microbundle of $W$ over $B$. Thus, a \emph{homotopy} of linearizations of $M$ is given by a fiberwise linearization of $M\times I$ over $I$.

We note that a microbundle morphism $\ll:VM\to EM$ carries the same information as an exponential map $\mu : U \to M'$ where $U$ is a neighborhood of the zero section of $VM$. As we remarked already, such a structure exists if $M$ is a smooth manifold. Then we have a smooth exponential map (inverse to orthogonal projection in $\RR^\infty$)
\[
	\mu_M: U(M)\to M'
\]
where $U(M)$ is a neighborhood of the zero section in $TM$. This gives a microbundle morphism $TM\to EM$ making $(M,TM)$ into a linearized topological manifold. We call this the \emph{canonical linearization} of $M$.

There is a problem that the topological derivative of a smooth map $M_0\to M_1$ is not covered by a linear map of canonical linearizations. So, instead we use smooth linearizations. A linearization (or fiberwise linearization) $\ll:VM\to EM$ of a smooth manifold $M$ will be called \emph{smooth} if $VM$ has a smooth structure compatible with the linear structure so that $\ll$ is a diffeomorphism in some neighborhood of the zero section. We note that the smooth structure on $VM$ is unique if it exists and the comparison map $\psi$ between any two smooth linearizations is also smooth. Note that the derivative of $\psi$ along the zero section gives an isomorphism of vector bundles $V\to V'$. Therefore, any smooth linearization of a smooth manifold is canonically isomorphic to its tangent bundle as a vector bundle.

\begin{prop} For any compact smooth manifold $M$,
the space of smooth linearizations $\mu:TM\to EM$ with fixed derivative along the zero section is convex and thus contractible.\qed
\end{prop}

Suppose $M_0, M_1$ are smooth manifolds with canonical linearizations $\ll_i:TM_i\to EM_i$ and $f:M_0\to M_1$ is a diffeomorphism with tangent map $Tf:TM_0\to TM_1$. Then the map germs
\[
	\ll_0,\ Ef^{-1}\circ\ll_1\circ Tf:TM_0\to EM_0
\]
are smooth linearizations of $M_0$ with the same derivative, namely the identity, along the zero section. Therefore, there is a 1-parameter family of smooth linearizations $\mu_t:TM_0\to EM_0$ all having the same derivative going from $\mu_0=\ll_0$ to $\mu_1=Ef^{-1}\circ\ll_1\circ Tf$. This is an example of a (smooth) tangential homeomorphism of smooth manifolds.

%%%\newpage
%-------------------------------------------------------------------------------------%-------------------------------------------------------------------------------------
%		subsub section {Stabilizing linearizations}
%-------------------------------------------------------------------------------------%-------------------------------------------------------------------------------------

\subsubsection{Stabilizing linearizations}

Suppose that $(M,VM)$ is a linearized manifold with exponential map germ $\mu:VM\to M'$. Then any vector bundle $p:L\to M$ will be seen to have an induced linearization on the (noncompact) manifold $L$. We will usually restrict to a disk bundle $D(L)$ which is compact. 

We choose an extension $L'\to M'$ of the vector bundle $L$ to $M'$ and assume we have a Gauss map $\g:L'\to\RR^N$, i.e. a continuous map which is a linear monomorphism on each fiber. This gives a metric on $L'$ and allows us to take the $\e$-disk bundle $D_\e(L')$. Over any two points $x,y\in M'$ we also have a linear map between fibers of $L'$:
\[
	\pi^x_y:L_x\to L_y
\]
given by orthogonal projection in $\RR^N$. When $x=y$ this is the identity map on $L_x$. Therefore, for some neighborhood $U$ of $x$ in $M'$ we get an isomorphism of vector bundles $\pi^x_U:U\times L_x\cong L'|U$ given by
\[
	\pi^x_U(y,w)=\pi^x_y(w)\in L_y\subseteq L'|U.
\]

The vector bundle of the linearization of $L$ induced by $VM$ will be the pull-back $p^\ast(VM\oplus L)$ of the direct sum $VM\oplus L$. The exponential map on the restriction of $p^\ast(VM\oplus L)$ to the zero section $M\subset L$ is the map
\[
	\ov\mu:VM\oplus L\to L
\]
given on the fiber $V_xM\times L_x$ over $x\in M$ by $\ov\mu=\pi^x_U\circ (\mu_x\times id)$ or
\[
\ov\mu(v,w)=\pi^x_{\mu_x(v)}(w)
\]
for some neighborhood $U$ of $x$ in $M'$.

Since this construction is continuous on the input data, it also works for vector bundles $L$ over fiberwise linearized manifold bundles $(M,V\vv M)$ over a manifold $B$ to produce a linearization of the Euclidean bundle $V\vv M\oplus L$ over $M$.

\begin{prop} An extension of this exponential map to all of $L$ exists and is well-defined up to homotopy. Furthermore, if $M, V\vv M, L, p,\g$ are smooth then $\ov\mu$ and its extension to $L$ will be smooth.\qed
\end{prop}

Since $M\subset L$ is a deformation retract, this follows from the following important lemma.

\begin{lem}[Linearization extension lemma]\label{extending linearizations from the spine}
Suppose that $W\to B$ is a topological manifold bundle and $K\subseteq W$ is a fiberwise deformation retract of $W$. Then any linearization $V$ of $E\vv W|K$ extends to all of $W$ and any two such extensions are homotopic rel $K$. Furthermore, if $W,V$ are smooth, then this extension will also be smooth.
\end{lem}

\begin{proof}
Choose a fiberwise deformation retraction $r_t:W\to W$ of $W$ to $K$. Then $r_t$ is covered by a deformation retraction $\tilde{r}_t:E\vv W\to E\vv W$ of the Euclidean bundle $E\vv W$ to $E\vv W|K$ which we can take to be fiberwise smooth in the smooth case and $\tilde{r}_1$ gives an isomorphism $(\tilde{r}_1)_\ast$ between $E\vv W$ and the pull-back $r_1^\ast(E\vv W|K)$ of $E\vv W|K$ to $W$. If $\ll:V\to E\vv W|K$ is a linearization of $E\vv W|K$ then 
\[
	r_1^\ast V\xrarrow{r_1^\ast(\ll)} r_1^\ast (E\vv W|K)\xrarrow{(\tilde{r}_1)_\ast^{-1}} E\vv W
\]
is a linearization of $E\vv W$ which will be fiberwise smooth in the smooth case and $(\tilde{r}_t)_\ast^{-1}\circ r_t^\ast(\ov\ll)$ is a deformation of any linearization $\ov\ll:\ov{V}\to E\vv W$ extending $\ll$ to this one.
\end{proof}

%%%\newpage
%--------------------------------------------------------------------------------------------
%--------------------------------------------------------------------------------------------
%			subsub section{Tangential homeomorphisms}
%--------------------------------------------------------------------------------------------
%--------------------------------------------------------------------------------------------

\subsubsection{Tangential homeomorphisms}\label{Tangential homeo} Two linearizations of a manifold $M$ are equivalent if they lie in the same path component of the space of linearizations of $M$, in other words there is a fiberwise linearization of $M\times I$ which agrees with these linearizations at the endpoints. Two linearized manifolds are equivalent if there is a homeomorphism between them so that the linearization of one manifold is equivalent to the pull back of the linearization of the other manifold. We will make these definitions more precise and extend them to manifold bundles over $B$.

%--------------------------------------------------------------------------------------------
%			sub- sub-section   {tangential homeomorphism}
%--------------------------------------------------------------------------------------------

\subsubsub{i}{tangential homeomorphism}

By a \emph{tangential homeomorphism} between linearized manifolds we mean a triple:
\[
	(f,Vf,\mu_t): (M_0,VM_0,\ll_0)\to (M_1,VM_1,\ll_1)
\]
where
\begin{enumerate}
\item $f:M_0\to M_1$ is a homeomorphism,
\item $Vf:VM_0\to VM_1$ is a nonsingular linear map over $f$, and
\item $\mu_t:VM_0\to EM_0$ is a one parameter family of linearizations of $M_0$ going from $\mu_0=\ll_0$ to $\mu_1=Ef^{-1}\circ\ll_1\circ Vf$:
\[
\xymatrix{VM_0\ar[d]_{\ll_0=\mu_0}&
VM_0\ar[r]^{Vf}\ar[d]^{\mu_1}& VM_1\ar[d]^{\ll_1}\\
EM_0&EM_0 \ar[r]^{Ef} & EM_1
}
\]
\end{enumerate}
A fiberwise tangential homeomorphism between fiberwise linearized manifolds bundles is defined similarly.

%--------------------------------------------------------------------------------------------
%			sub-sub-section {deformation of linearized manifolds}
%--------------------------------------------------------------------------------------------

\subsubsub{ii}{deformation of linearized manifolds}

We will see that a tangential homeomorphism $(M_0,VM_0)\to (M_1,VM_1)$ is equivalent to a one parameter family of linearized manifolds $(M_t,VM_t)$ going from $(M_0,VM_0)$ to $(M_1,VM_1)$. Then we define two tangential homeomorphisms to be \emph{isotopic} if the corresponding paths in the space of linearized manifolds are homotopic fixing the endpoints.

Suppose that $W\to I$ is a compact manifold bundle over the unit interval together with a fiberwise linearization $V\vv W\to E\vv W$. This is equivalent to the one parameter family of linearized manifolds $(M_t,VM_t)$ where $M_t=W|t$ with linearization $VM_t=V\vv W|t\to EM_t=E\vv W|t$.

Note that a tangential homeomorphism 
$(f,Vf,\mu_t):(M_0,VM_0)\to (M_1,VM_1)
$ also gives a one parameter family of linearized manifolds 
\[
(M_t,VM_t)=(M_0,\mu_t:VM_0\to EM_0)
\]
going from $(M_0,VM_0)$ to $(M_1,VM_1)$ if we identify $(M_0,VM_0,\mu_1)\cong (M_1,VM_1,\ll_1)$ via the isomorphism $(f,Vf)$. 

Conversely, we have the following.

\begin{prop}
A one parameter family of linearized manifolds $(M_t,VM_t)$ gives a tangential homeomorphism $(M_0,VM_0)\simeq (M_1,VM_1)$ which is uniquely determined up to a contractible choice. Furthermore, all tangential homeomorphisms are given in this way.
\end{prop}

\begin{proof} Since bundles over $I$ are trivial, there exist homeomorphisms $f_t:M_0\to M_t$ equal to the identity for $t=0$ covered by nonsingular vector bundle maps $Vf_t:VM_0\to VM_t$ giving us a one parameter family of linearizations 
$
	\mu_t:VM_0\to EM_0
$
making the following diagram commute.
\[
\xymatrix{
VM_0\ar[r]^{Vf_t}\ar[d]^{\mu_t}& VM_t\ar[d]^{\ll_t}\\
EM_0 \ar[r]^{Ef_t} & EM_t
}
\]
So $(M_t,VM_t)$ is isomorphic to $(M_1,\eta_t:VM_1\to EM_1)$ by $(g_t,Vg_t)$ for every $t\in I$. This in turn gives an equivalence of linearized manifolds 
\[
	(g_0,Vg_0,\eta_t\circ Vg_0):(M_0,VM_0)\simeq(M_1,VM_1).
\]
The choices that we made are the product structures for bundles over $I$. The space of such product structures is contractible.

If we start with a tangential homeomorphism and take the corresponding one parameter family of linearized manifolds then this construction recovers the original tangential homeomorphism. Therefore, the construction gives all tangential homeomorphisms.
\end{proof}

\begin{cor}\label{tangential homeos are equivalences}
Every tangential homeomorphism is invertible.
\end{cor}

\begin{proof}
Take the corresponding one parameter family of linearized manifolds and run it backwards. Then the composition of the tangential homeomorphism with its inverse is isotopic to the identity since the composition of the corresponding path and its inverse is homotopic fixing its endpoints to the constant path.
\end{proof}

%---------------------------------------------------------------------------------------
%			sub-sub-section {example}
%---------------------------------------------------------------------------------------

\subsubsub{iii}{example}

Suppose that $W$ is a smooth manifold bundle over $B$ which is a topological manifold bundle over $B\times I$ so that over $B\times 0$ and $B\times 1$ we have smooth bundles $M_0\to B$ and $M_1\to B$. Then we will obtain fiberwise tangential homeomorphisms $M_0\times I\to M_1\times I$ and $M_0\times I\to W$ as bundles over $B$.

(a) The first has the form
\[
	(f,V\vv f,\mu_t\vvv):(M_0\times I,T\vv (M_0\times I))\to(M_1\times I,T\vv (M_1\times I))
\]
with terms defined below. (Note that $T\vv (M_0\times I)=(T\vv M_0\oplus \e^1)\times I$ and similarly for $M_1\times I$.)

Since $W$ is a bundle over $B\times I$, $W$ is fiberwise homeomorphic to $M_0\times I$ over $B\times I$. This gives a homeomorphism
\[
	f_t:M_0\to M_t=\pi^{-1}(t)
\]
of bundles over $B$ and we let 
\[
f=f_1\times id_I:M_0\times I\to M_1\times I.
\]
The smooth structure on $W$ gives a linearization of the stabilized Euclidean bundle of $M_t$:
\[
	\mu_t:T\vv W|M_t\to E\vv M_t\oplus\e^1%=E\vv(M_t\times I)|(M_t\times 0)
\]
This gives a fiberwise linearization of the bundle $M_t\times I$ over $B$ which is smooth for $t=0,1$. The one parameter family of linearized manifolds $(M_t\times I, T\vv W|M_t\times I)$ gives a fiberwise tangential homeomorphism of the stabilizations
\[
	M_0\times I\to M_1\times I.
\]

(b) The second fiberwise tangential homeomorphism $M_0\times I\to W$ over $B$ is given as follows.

First, we add an external collar $M_0\times [-1,0]$ to the bottom of $W$. This gives $W^+=W\cup_{M_0}M_0\times[-1,0]$ a bundle over $B$ which is fiberwise diffeomorphic to $W$. But $W^+$ has a new projection map $\pi:W^+\to [-1,1]$. Let $W_t, t\in [0,1]$, be the 1-parameter family of topological manifolds given by $W_t=\pi^{-1}[-1,t]$. Since $W_t$ is topologically embedded in the smooth manifold bundle $W^+$ of the same dimension, $W_t$ obtains a linearization from the smooth linearization of $W^+$. The linearization is smooth for $t=0,1$ since $W_0,W_1$ are smooth submanifolds of $W^+$. This gives a 1-parameter family of linearized topological manifold bundles going from $W_0\cong M_0\times I$ to $W_1=W^+\cong W$ as claimed.

%%%\newpage
%-------------------------------------------------------------------------
%	subsubsection   tangential homeomorphism of smooth manifolds
%-------------------------------------------------------------------------

\subsubsection{Tangential homeomorphism of smooth manifolds}

By the discussion in the last subsection, a tangential homeomorphism between two smooth manifolds (with canonical linearizations) is given up to contractible choice by a fiberwise linearized topological bundle over $I$ which is smooth over the end points. We call this a ``tangential (topological) concordance'' between the two smooth manifolds. To avoid repetition, we give the formal definition only for bundles.

\begin{defn}
By a \emph{fiberwise tangential concordance} between two smooth manifold bundles $M_0\to B$, $M_1\to B$ over the same base we mean a linearized topological manifold bundle $W\to B\times I$ so that $W_0=W|B\times 0$ and $W_1=W|B\times1$ are smooth fiberwise linearizations of (the underlying topological manifold bundles of) $M_0$ and $M_1$.
\end{defn}

When we represent a fiberwise tangential concordance by a fiberwise tangential homeomorphism $(f,Vf,\mu_t)$ we would like to say that we can choose $Vf:T\vv M_0\to T\vv M_1$ to be smooth. However, this is not possible without changing the smooth structure of $T\vv M_0$ since $Vf$ is a map over the continuous map $f$ which is not, in general, homotopic to a diffeomorphism.

\begin{prop}
If the vertical tangent bundle of $M_0$ is trivial then any fiberwise tangential concordance between $M_0$ and $M_1$ is represented by a fiberwise tangential homeomorphism $(f,Vf,\mu_t)$ where $Vf:T\vv M_0\to T\vv M_1$ is smooth, using the smooth structure $T\vv M_0\cong M_1\times\RR^n$ and the given smooth structure on $T\vv M_1$. Furthermore the space of all such tangential homeomorphisms representing the same tangential concordance is contractible.
\end{prop}

\begin{proof}
The space of smooth nonsingular bundle maps $M_1\times \RR^n\to T\vv M_1$ over $M_1$ is homotopy equivalent to $Map(M_1,O(n))$ which is homotopy equivalent to the space of nonsingular continuous bundle maps $M_0\times\RR^n\to T\vv M_1$ over $f$.
\end{proof}

In fact, we can choose $f:M_0\to M_1$ to be a smooth embedding on the core of $M_0$ and we can choose $Vf$ to be smooth over that core.

% Section:

%\newpage
%%%%%%%%%%%%%%%%%%%%%%%%%%
%
%                subSection  {Tangential smoothing}
%
%%%%%%%%%%%%%%%%%%%%%%%%%%

\subsection{Tangential smoothing}\label{subsec14}

Given a linearized topological manifold $(M_0,VM_0)$, a \emph{tangential smoothing} of $(M_0,VM_0)$ is a smooth manifold $M_1$ together with a tangential homeomorphism $(M_0,VM_0)\simeq(M_1,TM_1)$. In this section we express this as a point in the homotopy fiber of a map between moduli spaces of manifolds with smooth and linear structures.

%%%\newpage
%--------------------------------------------------------------------------------------------
%--------------------------------------------------------------------------------------------
%			subsub section {space of linearized manifolds}
%--------------------------------------------------------------------------------------------
%--------------------------------------------------------------------------------------------

\subsubsection{Space of linearized manifolds}

Let $\widetilde S^t(n)$ be the set of all linearized $n$-manifolds.  For concreteness, we take these to be triples $(M,V,\mu)$ where $M\subseteq\RR^\infty$ is a compact topological $n$-manifold embedded in $\RR^\infty$, $\pi:V\to M$ is an $n$-plane bundle over $M$ and $\mu: V\to M'=M\cup C$ is a topological exponential map. As before, $C=\d M\times(-1,0]$ is the standard external collar for $M$.

Let $\widetilde\cS_\bullet^t(n)$ be the simplicial set whose $k$-simplices are continuous $\Delta^k$ families of linearized $n$-manifolds. This is a space which lies between the moduli spaces $\cS_\bullet^d(n)$ and $\cS_\bullet^t(n)$ in the sense that the simplicial forgetful map $\f:\cS_\bullet^d(n)\to\cS_\bullet^t(n)$ factors through $\widetilde\cS_\bullet^t(n)$:
\[
\xy
\xymatrix{
&\widetilde\cS_\bullet^t(n)\ar[dr]^{\psi^t}\\
\cS_\bullet^d(n)\ar[ur]^{\widetilde\f}\ar[rr]^\f&&\cS_\bullet^t(n)
}
\endxy
\]
The second map $\psi^t:\widetilde\cS_\bullet^t(n)\to \cS_\bullet^t(n)$ is the simplicial forgetful map given by projection to the first coordinate: $\psi^t(M,V,\mu)= M$. But, we are mainly interested in the first map
\[
	\widetilde\f:\cS_\bullet^d(n)\to \widetilde\cS_\bullet^t(n).
\]
This simplicial map is defined by taking a $\Delta^k$ family of smooth manifolds to the underlying family of topological manifolds with canonical linearizations.

%%%\newpage
%--------------------------------------------------------------------------------------------
%--------------------------------------------------------------------------------------------
%			subsub section {homotopy fiber of $\widetilde\f$}
%--------------------------------------------------------------------------------------------
%--------------------------------------------------------------------------------------------

\subsubsection{Homotopy fiber of $\widetilde\f$}

If we write down the definition it will be obvious that the homotopy fiber of $\widetilde\f$ is the space of fiberwise tangential smoothings of a fixed linearized topological manifold.

\begin{defn} For any linearized topological manifold $(M,VM)$
let $\widetilde\cS_\bullet^{t/d}(M,VM)$ be the simplicial set whose $k$-simplices are fiberwise tangential smoothings of the trivial bundle $M\times\Delta^k\to \Delta^k$ with fiberwise linearization $\ll\times id:VM\times \Delta^k\to EM\times\Delta^k$.
\end{defn}

\begin{prop}\label{the homotopy fiber of the map Sd to lin St}
$\widetilde\cS_\bullet^{t/d}(M,VM)$ is the homotopy fiber of the forgetful functor
\[
	\widetilde\f:{\cS_\bullet^d}(n)\to\widetilde\cS_\bullet^t(n)
\]
over $(M,VM)\in \widetilde\cS_0^t(n)$.\qed
\end{prop}

More generally, if $(M,V\vv M)$ is a fiberwise linearized manifold bundle over $B$, we can define a space $\ttd{B}(M)$ of fiberwise tangential smoothings of $(M,V\vv M)$. This is the homotopy fiber of the map
\[
	\widetilde\f_\ast:|{\cS_\bullet^d}(n)|^B\to|\widetilde\cS_\bullet^t(n)|^B
\]
over the map $B\to |\widetilde\cS_\bullet^t(n)|$ classifying the linearized bundle $(M,V\vv M)$. We are interested in $\pi_0$ of this space. However, this set may be empty. We need at least one smoothing $(M_0,V\vv M_0)$ to make it nonempty. This smoothing can be used as a base point and the other smoothings will be call ``exotic (fiberwise tangential) smoothings.''

\begin{defn}\label{def: fiberwise tangential smoothings}
 If $M_0\to B$ is a smooth manifold bundle, then an \emph{exotic (fiberwise tangential) smoothing} of $M_0$ is defined to be another smooth manifold bundle $M_1\to B$ together with a fiberwise tangentially homeomorphism to $M_0$.
\end{defn}

If $A$ is a submanifold of $\d B$ and $P$ is a smooth subbundle of $M$ then we define $\ttd{B,A}(M,P)$ to be the (simplicial) subspace of $\ttd B(M)$ consisting of fiberwise tangential smoothings of $(M,V\vv M)$ which are equal to the given smoothing on $P\cup M_A$ where $M_A=p^{-1}(A)$.

%%%\newpage
%--------------------------------------------------------------------------------------------
%--------------------------------------------------------------------------------------------
%		sub	sub section stabilization
%--------------------------------------------------------------------------------------------
%--------------------------------------------------------------------------------------------

\subsubsection{Stabilization}
By stabilization we mean taking direct limit with respect to all linear disk bundles. In particular we replace the tangential smoothing space $\ttd{B,A}(M_0)$ with the \emph{stable tangential smoothing space}
\[
	\std{B}{A}(M_0):=\lim_\to
	\ttd{B,A}(D(\xi))
\]
where the direct limit is with respect to all linear disk bundles $\xi$ over $M_0$. 

Stabilization is used to make the vertical tangent bundle trivial. We take the disk bundle $D(\nu)$ of the vertical normal bundle $\nu$, the stably well-defined complement for the vertical tangent bundle. If we replace $M_0$ with $D(\nu)\times D^m$ with corners rounded then this new $M_0$ has a smooth spine $J_0=D(\nu)^-$ (where the $(\ )^-$ indicates removal of a small open collar neighborhood of the vertical boundary of $D(\nu)$) which is a smooth manifold bundle. It has a core $K_0$ which is fiberwise diffeomorphic to $M_0$ and whose complement $M_0-K_0$ is an internal collar neighborhood of the fiberwise boundary $\d\vv M_0$. Thus $M_0-K_0$ is diffeomorphic to $\d\vv M_0\times[0,1)$.

\subsubsub{i}{corners}

There is one problem: We need to know that corners can be rounded off in a canonical way. But, for our purposes, this is easy since any two ways of rounding off corners will clearly be tangentially homeomorphic and we have the following lemma.

\begin{lem}
Let $M_0,M_1$ be two smooth bundles over $B$ which are fiberwise diffeomorphic over $A\subset \d B$ and which are also fiberwise tangentially homeomorphic relative to $A$. Then their tangential smoothing spaces relative to $A$ are simplicially homotopy equivalent:
\[
	\ttd{B,A}(M_0)\simeq
	\ttd{B,A}(M_1).
\]
\end{lem}

\begin{proof}
The fiberwise tangential homeomorphism $M_0\to M_1$ gives a one parameter family of linearized topological manifold bundles $M_t$ going from $M_0$ to $M_1$ which we view as a path in the space of linearized topological manifolds. 

A point in $\ttd{B,A}(M_0)$ is a smooth bundle $M'$ together with a fiberwise tangential homeomorphism with $M_0$. Composing this with the given fiberwise tangential homeomorphism $M_0\to M_1$ we get a point in $\ttd{B,A}(M_1)$. The inverse morphism give a simplicial map the other way $\ttd{B,A}(M_1)\to \ttd{B,A}(M_0)$ and the composition is the identity by Corollary \ref{tangential homeos are equivalences}
\end{proof}

One important example was given in subsubsubsection \ref{Tangential homeo}(iii). Given a smooth bundle $W\to B$ which fibers topologically over $B\times I$ in such a way that the inverse image of $B\times 0$ and $B\times1$ are smooth bundles $M_0$ and $M_1$, then $W$ defines a tangential homeomorphism $M_0\times I\to M_1\times I$. In other words $M_1\times I$ becomes a point in the stable tangential smoothing space of $M_0$. We denote this by
\[
	top(W)=M_1\times I\in\std{B}{A}(M_0)
\]
This is represented by $M_1$ in the sense that $M_1$ and $top(W)=M_1\times I$ are stably equivalent.

\subsubsub{ii}{flat sides}\label{flat subsubsection}

Stabilization can also be given by the simple process of taking products with disks:
\[
	\std{B}{A}(M_0)=\lim_\to \ttd{B,A}(M_0\times D^n)
\]
The reason is that trivial disk bundles are cofinal in the directed system of all disk bundles over $M_0$. Since corners are not a problem, we can also use cubes $I^n$ instead of disks $D^n$.

Another cofinal system is given by the space of tangential smoothings $W$ of $M_0\times I^n\times I$ which are ``flat'' on $M_0\times I^n\times 0$ in the sense that the tangential homeomorphism
\[
	(f,V\vv f):M\times I^n\times I\to W
\]
is induced by a tangential homeomorphism
\[
	(\d_0f,V\vv\d_0f):M\times I^n\times 0\to \d_0W
\]
in a neighborhood of $M\times I^n\times 0$. This is only a restriction on the tangential map $V\vv f$ and homotopy $\mu\vv$ since any homeomorphism $f$ as above induces a homeomorphism $\d_0f$ on $M\times I^n\times 0$. The ``flatness'' condition is that the maps 
\[
\mu\vv_t:V\vv(M\times I^n\times I)\to V\vv(W)
\]
should send $V\vv(M\times I^n\times 0)$ to $V\vv(\d_0W)$.

It is easy to see that ``flat on one side'' smoothings form a cofinal system. When we pass from a smoothing of $W$ to a smoothing of $W\times I$ we always have a tangential smoothing which is flat on one side (in fact on both sides). Similarly, when we stabilize ``flat on one side'' smoothings we first forget the flatness on one side then take the product with an interval.

\subsubsub{iii}{smoothness of the boundary}

Finally, we need the fact that, after stabilization, tangential smoothings of $M$ which are fixed on the vertical boundary $\d\vv M$ give the same thing as those which don't fix the boundary. We will formulate this more precisely and prove it in the next section.

Suppose that $\d B=\d_0B\cup \d_1B$ where $\d_0B$ and $\d_1B$ meet along their boundary $\d\d_0B=\d\d_1B$. The boundary of the total space $M$ is equal to $\d M=\d_0M\cup \d_1M$ where $\d_0M=\d\vv M\cup M_{\d_0B}$ be the union of the vertical boundary $\d\vv M$ of $M$ and the restriction $M_{\d_0B}$ of $M$ to $\d_0B$ and $\d_1M=M_{\d_1B}$. 

Recall that $\ttd{B,\d_0B}(M,\d\vv M)$ is the space of tangential smoothings of $M$ which are fixed on $\d_0M$. (Def. \ref {def: fiberwise tangential smoothings}) This is a subspace of $\ttd{B,\d_0B}(M)$.

\begin{prop}\label{boundary does not matter}
After stabilization we get a homotopy equivalence:
\[
	\lim_\to \ttd{B,\d_0B}(D(M),D(\d\vv M))\simeq\lim_\to \ttd{B,\d_0B}(D(M))
	\]
	where both limits are with respect to all linear disk bundles $D(M)$ over $M$ and $D(\d\vv M)$ is the restriction of $D(M)$ to $\d\vv M$.
\end{prop}

%\end{document}

%\newpage
%%%%%%%%%%%%%%%%%%%%%%%%%%%%%%%%%%%
%
%				subSection  {Smoothing theorems}
%
%%%%%%%%%%%%%%%%%%%%%%%%%%%%%%%%%%%

\subsection{Smoothing theorems}\label{subsec15}

Given a smooth bundle $p:M_0\to B$ and a smooth submanifold $A$ of $\d B$, we would like to determine the set of all isotopy classes of exotic fiberwise smoothings of $M_0$ which are equal to the given smoothing over $A$. By definition this is $\pi_0$ of the space 
$
	\ttd{B,A}(M_0)
$ of all tangential smoothings of $M_0$:
\[
	(f,V\vv f,\mu_t\vvv):(M_0,T\vv M_0,\ll_0)\to (M_1,T\vv M_1,\ll_1)
\]
which are trivial over $A$.
This lies in the null component of $ \ttd{B,A}(M_0)$ if and only if $f:M_0\to M_1$ is isotopic to a fiberwise diffeomorphism and $\mu_t\vvv$ is isotopic to a family of smooth fiberwise linearizations of $M_0$.

When we stabilize $M_0$ will contain a core $K_0$ which is $M_0$ minus an internal fiberwise collar and a spine $J_0$ which is a high codimensional submanifold of $K_0$. Thus $K_0$ will be a smooth manifold bundle diffeomorphic to $M_0$ and both $K_0$ and $J_0$ will be fiberwise deformation retracts of $M_0$.

It will follow from standard immersion theory that $(f,V\vv f)$ is isotopic to a smooth embedding on the core $K_0$ in such a way that $\mu_t\vvv$ becomes isotopic to a smooth linearization over the core. What will remain is the question of smoothability of the internal collar. The extension of the smooth linearization is automatic by the linearization extension lemma \ref{extending linearizations from the spine}. So we have a classical smoothing problem whose obstruction space is well known to be a homology theory.

%%%\newpage
%--------------------------------------------------------------------------------------------
%--------------------------------------------------------------------------------------------
%		sub	sub section {smoothing of disk bundles}
%--------------------------------------------------------------------------------------------
%--------------------------------------------------------------------------------------------

\subsubsection{Smoothing of disk bundles} We will go over the simplest example: disk bundles.

%-------------------------------------------------------------------------------------
%			sub- sub-section   {the problem}
%-------------------------------------------------------------------------------------

\subsubsub{i}{the problem}
We are given two smooth bundles $M_0, M_1$ over $B$ with fiber $D^n$, a smooth $n$-dimensional disk, which are fiberwise diffeomorphic over a submanifold $A$ of $B$ and a fiberwise tangential homeomorphism
\[
	(f,V\vv f,\mu_t\vvv):(M_0,T\vv M_0,\ll_0)\to (M_1,T\vv M_1,\ll_1)
\]
which agree with the given diffeomorphism over $A$. 
We want to find an isotopy of $f$ rel $A$ to a diffeomorphism over $B$ in a way compatible with the tangential data given by $V\vv f,\mu_t\vvv$.

%-------------------------------------------------------------------------------------
%			sub- sub-section   {spines and cores}
%-------------------------------------------------------------------------------------

\subsubsub{ii}{spines and cores}

The first step is to choose a spine. The spine of a disk is any interior point and the spine of a disk bundle is any section which lies in the interior. Let $s_0,s_1$ be smooth sections of $M_0,M_1$ with images $J_0,J_1$ in the fiberwise interiors so that $s_1=f\circ s_0$ over $A$. We will choose a small standard disk bundle neighborhood $K_0$ of the image $J_0$ of $s_0$.

Next, we deform $f:M_0\to M_1$ so that it takes $J_0$ to $J_1$. Since the fibers are contractible, the sections $s_1$ and $f\circ s_0$ of $M_1$ are homotopic rel $A$. We can use the topological ambient isotopy theorem to extend this to an isotopy of $f$ rel $A$ to a homeomorphism taking $J_0$ to $J_1$. (However, this example of the ambient isotopy theorem is an easy exercise.)

\begin{lem}
Any topological isotopy of $f$ rel $A$ can be extended to $V\vv f$ and $\mu_t\vvv$ to give an isotopy of the tangential homeomorphism $(f,V\vv f,\mu_t\vvv)$.
\end{lem}

\begin{proof}
Consider the tangential homeomorphism as a 1-parameter family of linearized manifolds $(M_t,VM_t,\mu_t)$ together with a family of homeomorphisms $f_t:M_0\to M_t$ which is the identity on $B\times 0\cup A\times I$. To prove the lemma we take the same family of linearized manifold bundles with a new family of homeomorphisms.
\end{proof}

%-------------------------------------------------------------------------------------
%sub- sub-section   {extending the smooth structure to the core}
%-------------------------------------------------------------------------------------

\subsubsub{iii}{extending the smooth structure to the core}

Let $(M_t,V\vv M_t,\mu_t\vvv,J_t)$ be the 1-parameter family of linearized manifolds over $B$ with a given choice of spine which is smooth over $B\times 0\cup A\times I\cup B\times 1$.

Let $V_t=V\vv M_t|J_t,E_t=E\vv M_t|J_t$ considered as bundles over $B$ and let $\mu_t:V_t\to E_t$ the microbundle isomorphism given by the restriction of $\mu_t:V\vv M_t\to E\vv M_t$ to $J_t$. Then $V_t$ will be a vector bundle over $B$ and $E_t$ will be a Euclidean bundle over $B$ which is fiberwise homeomorphic to a neighborhood of $J_t$. Since continuous isomorphisms of smooth vector bundles are isotopic to smooth vector bundle isomorphisms, we can choose a family of vector bundle isomorphisms $V_0\cong V_t$ which is the identity for $t=0$ and smooth for $t=1$ and we can do this relative to $A$. This gives a smooth structure on $V_t$ for all $t$ which agrees with the smooth structure over $A$ and over the endpoints.

Using the microbundle isomorphism $\mu_t:V_t\to E_t$ we get a family of smoothings for a neighborhood of $J_t$ in $M_t$ which is compatible with $V\vv M_t$ over $J_t$. By definition, $\mu_t$ will be a smooth linearization of $E_t$. By the linearization extension lemma \ref{extending linearizations from the spine}, we can extend this to a new fiberwise linearization of $M_t$ which is smooth in a neighborhood of $J_t$ (and everywhere where it was already smooth). Furthermore this new linearization will be isotopic to the old one.

%-------------------------------------------------------------------------------------
%			sub- sub-section   {smoothing the collar}
%-------------------------------------------------------------------------------------

\subsubsub{iv}{smoothing the collar}

The situation is the following. We have a 1-parameter family of linearized manifold bundles $(M_t,V\vv M_t,\mu_t\vvv)$ together with a smoothing over the core $K_t$ which is a disk bundle in the interior of $M_t$. We also have a smoothing over $B\times 0\cup A\times I\cup B\times1$ which is compatible with the linearization.

The key point is that smoothing is \emph{excisive} (Proposition \ref{smoothing is excisive}). Therefore, we may remove the interior of the core $K_t$. If we stabilize $M_t$ once more, replacing it with $M_t\times[-1,1]$, we will have the smooth subbundle $K_t\times [-1,0]$ which meets the boundary. Then, after excising the interior of this new subbundle and rounding off the corners, we get $M_t\times[-1,1]-\interior K_t\times [-1,0)$ which is a topological $h$-cobordism bundle over $B$ whose fibers are $h$-cobordisms of $D^n$ and therefore homeomorphic to $D^n\times I$ which have a smooth structure on the base $D^n\times 0$ and sides $\d D^n\times I$ and over $B\times 0\cup A\times I\cup B\times 1$.

This can be rephrased as follows. We have a continuous mapping of pairs
\[
	(B\times I,B\times 0\cup A\times I\cup B\times 1)\to (\cC ob^t(D^n),\cC ob^d(D^n))
\]
where $\cC ob^t(M)$ is the space of topological $h$-cobordisms $W\subseteq \RR^\infty$ of $M$ which are fixed on the base $M\times 0$ and the boundary, $\d M\times I$ and $\cC ob^d(M)$ is the space of pairs $(W,\a)$ where $W\in\cC ob^t(M)$ and $\a$ is a smoothing of $W$ which agrees with a fixed standard smoothing on $M\times 0\cup\d M\times I$. These spaces are topologized as geometric realizations of simplicial subsets of $\cS^t(n+1)$ and $\cS^d(n+1)$.

We use the following facts:\begin{enumerate}
\item $\cC ob^t(D^n)$ is contractible by the Alexander trick.
\item $\cC ob^d(D^n)$ is an $n$-fold loop space since it has an action of the little $n$-cubes operad.
\item {The smooth structure is fixed over $B\times 0$.}
\end{enumerate}
Therefore, we can trivialize the smooth structure over $B\times 0$ by multiplying by its inverse. The smooth structure over $A\times I$ can also be made trivial in homotopy unique way. The map to $\cC ob^h(D^n)$ contains no homotopy information. So, we are reduced to a map
\[
	B/A\to \cC ob^d(D^n)
\]
where $B/A$ means smashing $A$ to a point.

\begin{thm}
If $M_0$ is a smooth disk bundle over $B$ and $A$ is a submanifold of $B$ then the space of stable fiberwise tangential smoothings of $B$ rel $A$ is homotopy equivalent to the space of all pointed maps
\[
	B/A\to \cH(\ast)
\]
where
\[
	\cH(\ast):=\colim\cC ob^d(D^n)
\]
is the stable smooth $h$-cobordism space of a point.\qed
\end{thm}

%-----------------------------------------------------------------------------------
\subsubsub{v}{higher torsion calculation}
%--------------

We use the well-known fact that $\cH(\ast)$ is rationally homotopy equivalent to $BO$. This was first shown by Farrell and Hsiang and later Hatcher gave an explicit map $G/O\to \cH(\ast)$ and conjectured that it was nontrivial. This was first proved by B\"okstedt and later by Igusa using higher Reidemeister torsion. See \cite{I:Axioms0} for an elementary explanation of this.

We note that $\std{B}{A}(B,0)$ is the space which classifies stable exotic smooth structures on linear disk bundles over $B$ which are trivial over $A$. ($B\to B$ is the unique $D^0$ bundle over $B$ with trivial vertical tangent bundle $0$.) The theorem shows that
\[
	\std{B}{A}(B)\simeq Map(B/A,\cH(\ast)).
\]

\begin{cor}\label{smooth disk bundles and IK torsion}
$\pi_0\std{B}{A}(B)$ is an abelian group and we have an isomorphism
\[
	\t^\IK:\pi_0\std{B}{A}(B)\otimes\RR\cong \bigoplus_{k>0}H^{4k}(B,A;\RR)
\]
given by sending any smooth disk bundle $E\to B$ which is linear over $A$ and any tangential homeomorphism of $E$ to a linear disk bundle to
\[
	\t^\IK(E)=\sum \t_{2k}^\IK(E)\in \bigoplus_{k>0}H^{4k}(B,A;\RR)
\]
\end{cor}

We note that $\t^\IK(E)$ ignores the tangential data. If we took a different axiomatic higher torsion theory (such as the nonequivariant higher analytic torsion) we would need to subtract the higher torsion of the linear bundle for which $E$ is an exotic smooth structure.

%%%\newpage
%--------------------------------------------------------------------------------------------
%--------------------------------------------------------------------------------------------
%			subsub section {immersion theory}
%--------------------------------------------------------------------------------------------
%--------------------------------------------------------------------------------------------

\subsubsection{immersion theory} 

We are now looking at a stabilized exotic tangential smoothing $(W,V\vv W)$ of the bundle $M_0\to B$ which is given by a tangential homeomorphism $(f,V\vv f,\mu)$ between the two smooth bundles $M_0,M_1$ over $B$. After stabilizing $M_0$ has a high codimensional \emph{spine} $J_0$ which is a smooth submanifold bundle of the fiberwise interior of $M_0$ which is a fiberwise deformation retract of $M_0$. Furthermore $J_0$ can be chosen to have trivial vertical normal bundle in $M_0$. We also have a \emph{core} $K_0$. In this case we choose $K_0$ to be a small tubular neighborhood of $J_0$. We also need to assume that $J_0$ contains a submanifold bundle $P_0$ so that $f$ is already a smooth embedding in a neighborhood of $P_0$ and $W$ is smooth in a neighborhood of $P_0\times I\cup W_A$ where $A$ is a submanifold of $B$.

By standard immersion theory (\cite{Hirsch59},\cite{GromovEliashberg71}), there is a fiberwise immersion $g:K_0\to M_1$ over $B$ which is regularly fiberwise homotopic to $(f,V\vv f)$ restricted to $K_0$. Since the spine has a high codimension, we have by transversality that $g$ is an embedding on $J_0$. By replacing the core $K_0$ with a smaller core we may also assume that $g$ is an embedding on $K_0$. We may also assume that $g$ is equal to $f$ in a neighborhood of $P_0$ and over $A$.

Immersion theory tells us that that $f,g$ are fiberwise homotopic (fixing a neighborhood of  $P_0$) by a one parameter family of continuous maps $h_t:K_0\to M_1$ over $B\rel A$ and the fiberwise derivative $T\vv g:T\vv K_0\to T\vv M_1$ is homotopic through nonsingular linear maps $V\vv h_t$ to the vector bundle map $V\vv f$. Given any $\e>0$ we can choose the immersion $g$ and the homotopy $h_t$ to be within $\e$ of $f$ and that $h_t=f=g$ over $A$ and near $P_0$ for all $t$.

\begin{prop}
After stabilization, we can choose $(h_t,V\vv h_t)$ so that $h_t:K_0\to M_1$ is a fiberwise topological embedding for all $t\in I$.
\end{prop}

\begin{proof} First, we can reduce the structure group of the linear bundle $V\vv W$ over $W$ to $O(n)$. Then we get a linear $\e$-disk bundle $D\vv W$ over $W$ which, as a bundle over $B\times I$, is linearized with vertical Euclidean bundle isomorphic to $V\vv W\oplus V\vv W$. This is a stabilization. So, it suffices to prove the theorem for $D\vv W$.

The idea of the proof is the following. The interior of the disk bundle $D\vv W$ is homeomorphic to the total space of the vertical tangent Euclidean bundle $E\vv W$. So, the corresponding tangential homeomorphism, when restricted to the core of $D\vv M_0$, is given by the topological vertical derivative $E\vv f$ of $f$. The topological derivative $E\vv g$ of $g$ is a smooth embedding. Therefore, it suffices to show that $E\vv f$ and $E\vv g$ are homotopic through fiberwise topological embeddings.

The total space of $E\vv K_0$ is the set of all pairs $(x,y)$ in the same fiber of $M_0$ over $B$ so that $x\in K_0$ and $y\in B_\e\vv (x)$ where $B_\e\vv (x)$ is the open $\e$-ball neighborhood of $x$ in the fiber of $M_0\to B$. Inside of this space we have the following two subspaces where $\delta<<\e$ is the number so that $K_0=L_{2\delta}$ where $L_{\w}$ the open $\w$-neighborhood of the spine $J_0$: 
\[
	U_\delta= \{(x,y)\in E\vv K_0\st x\in L_\delta, d(x,y)<\delta\}
\]
\[
	U_{\delta/2}=\{(x,y)\in E\vv K_0\st y\in L_{\delta/2}, d(x,y)<\delta/2\}
\]
Then $U_{\delta/2}\subseteq U_\delta$. We will show that the restrictions of $E\vv f,E\vv g$ to $U_{\delta/2}$ are isotopic and that the isotopy agrees with the given homotopy.

The embedding $E\vv f$ maps $(x,y)\in U_\delta$ to $(f(x),f(y))\in E\vv M_1$ by definition. For every fixed $x\in L_\delta$, this mapping sends $x\times B_\delta(x)$ to $f(x)\times E\vv f_x(B_\delta)$ by the mapping $E\vv f_x$ ($E\vv f$ restricted to the fiber of $E\vv M_0\to M_0$ over $x$) on the second factor. But the homotopy $V\vv h_t$ is an isotopy from $E\vv f$ to $E\vv g$. Therefore, the embedding $E\vv f$ is isotopic to the embedding which sends $x\times B_\delta(x)$ to $f(x)\times E\vv g_x(B_\delta)$ by the mapping $E\vv g_x$ on the second factor.

By definition of $E\vv g$, this new embedding is 
\[
	f\times g:(x,y)\to (f(x),g(y)).
\]
So, $E\vv f|U_\delta$ is isotopic to $(f\times g)|U_\delta$. By the same argument, $E\vv g|U_{\delta/2}$ is isotopic to $(f\times g)|U_{\delta/2}$. Since $U_{\delta/2}\subseteq U_\delta$, we get an isotopy from $E\vv f|U_{\delta/2} $ to $E\vv g|U_{\delta/2}$. And this isotopy will be fixed near $P_0$ and over $A$. As was already shown, these are the stabilized versions of $f$ and $g$ on a small tubular neighborhood of the spine of $M_0$. So, we are done.
\end{proof}

%%%\newpage

%-----------------------------------------------------------------------------------
%            subsub section {smoothing of the core} A.5.3
%-----------------------------------------------------------------------------------

\subsubsection{smoothing of the core}

\begin{thm}
There is no stable obstruction to finding a smoothing of the core of $W$. I.e., after stabilizing, the one parameter family of linearized topological manifolds bundles $(M_t,V\vv M_t,\mu_t\vvv)$ has a smoothing compatible with the linearization in some tubular neighborhood of the spine $J_t$. Furthermore, this smoothing will be equal to the given smoothing of $W$ on $P_0\times I\cup W_A$ if $W$ is already smooth on this set.
\end{thm}

\begin{rem}
By replacing $B$ with $B\times I$ and $A$ with $A\times I\cup B\times\{0,1\}$, we conclude that the smoothing of the core is unique up to homotopy.
\end{rem}

\begin{proof}
By the proposition above, we may assume that this 1-parameter family of linearized manifolds is given by a tangential homotopy equivalence $(f,V\vv f,\mu_t\vvv)$ where $M_0,M_1$ are smooth manifold bundles over $B$ which are diffeomorphic over $A$ and $f:M_0\to M_1$ is a smooth embedding on the core $K_0$ and $V\vv f$ is the vertical derivative of $f$ along $K_0$. Returning to the 1-parameter family of manifold bundles $(M_t,V\vv M_t)$, this implies that we have a continuous family of submanifolds $J_t\subseteq M_t$ which are smooth submanifolds for $t=0,1$ and these submanifolds have tubular neighborhoods which have product structures: $K_t\cong J_t^+\times D^n$. where $J_t^+$ is $J_t$ with an external closed collar attached.

After stabilization, we may assume that the vector bundle $V\vv M_t$ is trivial: $V\vv M_t\cong M_t\times\RR^{k+n}$ where $k$ is the dimension of the fiber of $J_t\to B$. The vertical tangent Euclidean bundle will also be trivial: $E\vv M_t\cong M_t\times \RR^{k+n}$ and the linearization is given by a family of microbundle morphisms $\mu_t:V\vv M_t\to E\vv M_t$ which is smooth for $t=0,1$ and on the restriction of $V\vv M_t$ to $V\vv K_t$. This is equivalent to a family of mappings
\[
	\widehat \mu_t:M_t\to Homeo(\RR^{k+n})
\]
By the first lemma below, we can assume, after stabilization, that this map has image in the subgroup of all homeomorphisms of $\RR^{k+n}$ having the form $g(x,y)=(x,g_x(y))$, i.e., they are $\RR^k$ families of homeomorphisms of $\RR^n$.

Along $J_t'$ the spine with an open external collar, we now have a linearization
\[
	\mu_t:V\vv M_t|J_t'\cong J_t'\times\RR^{k+n}\to E\vv M_t|J_t'\cong J_t'\times\RR^{k+n}
\]
which commutes with the projection to $J_t'\times\RR^k$. Now restrict this to the fiber over $J_t'\times 0$. This gives a family of linearizations
\[
	\eta:J_t'\times\RR^n\to J_t'\times \RR^n
\]
This is a map from a smooth linear bundle to a Euclidean bundle which, by the topological exponential map is homeomorphic to a neighborhood of $J_t$ in $M_t$. We can use this map to change the smooth structure in this neighborhood so that $\eta$ is a smooth map. By the second lemma below, we can deform the original linearization to a linearization which is smooth in a neighborhood of the spine. This will contain a somewhat smaller core but it is enough to prove the theorem.
\end{proof}

It remains to prove the two lemmas used in the theorem.

\begin{lem}
The image of the stabilization map $\s:Homeo(\RR^{k+n})\to Homeo(\RR^{k+n+k})$ given by $\s(f)(x,y,z)=(f(x,y),z)$ can be deformed into the subgroup of all homeomorphisms of $\RR^{k+n+k}$ having the form $g(x,y,z)=(x,g_x(y,z)))$. Furthermore this deformation will always send smooth maps to smooth maps.
\end{lem}

\begin{proof}
The deformation is given by rotation. Let $\r_\th$ be the linear and thus smooth automorphism of $\RR^{k+n+k}$ given by the matrix
\[
	\r_\th=\begin{pmatrix}
	\cos \th I_k &0& -\sin\th I_k\\
	0 & I_n & 0\\
	\sin\th I_k &0& \cos \th I_k
	\end{pmatrix}
\]
Then, $\r_\th\circ\s(f)\circ\r_{-\th}, 0\le\th\le \pi/2$ is the desired deformation.
\end{proof}

\begin{lem}
Let $G$ be the subgroup of $Homeo(\RR^{k+n})$ consisting of homeomorphisms $g$ having the form $g(x,y)=(x,g_x(y))$ and so that $g_0$ is a smooth diffeomorphism. Let $G_0$ be the subgroup consisting of diffeomorphisms of $\RR^{k+n}$ which lie in $G$. Then $G_0$ is a deformation retract of $G$.
\end{lem}

\begin{proof}
This is given by the Alexander trick:
\[
	g^t(x,y)=(x,g_{tx}(y))
\]
for $0\le t\le 1$. If $g$ is smooth then so is $g^t$.
\end{proof}

%-----------------------------------------------------------------------------------
%            subsub section {ignoring the boundary} A.5.4.
%-----------------------------------------------------------------------------------

\subsubsection{ignoring the boundary}

We are now ready to prove Proposition \ref{boundary does not matter} which says that, after stabilization, smoothings which are fixed on the vertical boundary and those which are not form homotopy equivalent spaces. The reason is that, after rounding corners, the core is diffeomorphic to the union of the core with the stabilized vertical boundary. So, after smoothing the core we cannot distinguish between the two spaces.

First, we use the ``flat on one side'' observation (subsubsection \ref{flat subsubsection}) to stabilize and have a flat side $\d_0W\subseteq W$ which is tangentially homeomorphic to $D(M)\times0\subseteq D(M)\times I$ where $D(M)$ is a linear disk bundle over $M$. On the flat side we can make $f$ smooth on the core $K_0\times 0$ using the theorem above. (Equivalently, we can smooth the core of $D(M)$ and then stabilize to get a smoothing of $K_0\times I$ and then forget the smoothing on all but $K_0\times0$. However, uniqueness up to homotopy of this second method is not as easy to see.)

By construction of the core $K_0$, the complement of $K_0$ in $D(M)$ is a product of $\d\vv D(M)$ with an interval. Therefore, the pair $(D(M)\times I,K_0)$ is, after rounding corners, diffeomorphic to $(D(M)\times I,D(M)\times 0)$ and $(W,K_0)$ is a tangential smoothing of that pair. If we apply the same construction to a tangential smoothing of $D(M)\times I$ which is fixed on $D(\d\vv M)\times I$, we make the homeomorphism $f$ smooth on $D(\d\vv M)\times I\cup K_0\times0$ where $K_0$ is a neighborhood of the spine $J_0\cong M$ which is the zero section of the disk bundle $D(M)$. But $D(\d\vv M)\times I$ is a disk bundle over $\d\vv M\times I$ which is an external collar for $M$ and $K_0$ is a disk bundle over $M$ of the same dimension. So, together they form a disk bundle over $M$ with an external collar (after rounding corners). Thus there is a diffeomorphism of $D(M)\times I$ with corners rounded which takes $D(\d\vv M)\times I\cup K_0\times I$ to $D(M)\times 0$, making the two stabilized tangential homeomorphisms equivalent. This proves the following extension of Proposition \ref{boundary does not matter}.

\begin{thm}
After stabilization we get homotopy equivalences:
\[
	\lim_\to \ttd{B,\d_0B}(D(M),D(\d\vv M))\simeq\lim_\to \ttd{B,\d_0B}(D(M))\simeq\lim_\to \cS^{t/d}_{B,\d_0B}(D(M)\times I,D(M)\times 0)
	\]
	where all three limits are with respect to all linear disk bundles $D(M)$ over $M$ and $D(\d\vv M)$ is the restriction of $D(M)$ to $\d\vv M$.

\end{thm}

We removed the tilde from the last version of stabilization since, by the linearization extension lemma \ref {extending linearizations from the spine}, the linearization of a smoothing of $(D(M)\times I,D(M)\times0)$ is unique up to contractible choice.

%-----------------------------------------------------------------------------------
%            subsub section {little cubes operad} A.5.5
%-----------------------------------------------------------------------------------

\subsubsection{little cubes operad}

Using the third form of stabilization given in the theorem above, we can see the infinite loop space structure on the stabilized smoothing space $\std{B}{\d_0B}(M)$. Recall that the space $C_k(n)$ of $k$ little $n$ cubes in $I^n$ is given by $k$ disjoint embeddings $y_i:I^n\to I^n$ which are given by affine linear maps $y_i(x)=a_ix+b_i$ where $a_i$ are positive real numbers and $b_i\in I^n$.
\[
	\a:C_k(n)\times \cS^{t/d}_{B,\d_0B}(D(M)\times I^{n+1},D(M)\times I^n)^k\to \cS^{t/d}_{B,\d_0B}(D(M)\times I^{n+1},D(M)\times I^n)
\]
\[
	\a(y; W_1,\cdots,W_k)= D(M)\times I^{n+1}\cup_{y} \coprod W_i
\]
where the base of each $W_i$ is attached to the top of $D(M)\times I^{n+1}$ using the map
\[
	1_{D(M)}\times y_i:\d_0 W_i= D(M)\times I^n\times0\to D(M)\times I^n\times 1
\]
and the resulting corners are rounded.

There is an easier way to describe the addition operation in the case when the supports of the exotic smooth structures are disjoint. An element of $\ttd{B,\d_0B}(M,C)$ is said to have {\bf support} in the closure of the complement of $C$ in $M$.

\begin{prop}\label{prop: addition on smooth structures is disjoint union}
The addition operation on the stable smoothing space $\std{B}{\d_0B}(M)$ given by the little cubes operad action described above is given unstably on smooth structures on $M$ with disjoint supports $S_i\subseteq M-\d\vv M$ by $W=\sum W_i$ which is equal to $M$ in the complement of $\coprod S_i$ and equal to $W_i$ on $S_i$.
\end{prop}

We will only apply this proposition to the group structure on $\pi_0\std{B}{\d_0B}(M)$. 

\begin{proof}
We will prove that the following diagram commutes up to homotopy for any $y\in C_k(n)$ where $C_i$ is the closure of the complement of $S_i$ in $M$.
\[
\xymatrix{
\prod \ttd{B,\d_0B}(M,C_i)\ar[d]\ar[r]^\Sig &
	\ttd{B,\d_0B}(M,\bigcap C_i)\ar[d]
	\\
\prod \std{B}{\d_0B}(M) \ar[r]^{\a(y;-)}& 
	\std{B}{\d_0B}(M)
	}%end xymatrix
\]
Note that, except for $\Sig$, these maps are only defined up to homotopy.

Suppose that $(M_i)$ is an element of the upper left corner. Thus $M_i$ is a tangential smoothing of $M$ over $(B,\d_0B)$ with support in $S_i$. If we stabilize by taking a product with a disk, we get $M_i\times D^N\times I$ which is a smooth structure on $M\times D^N\times I$ with support in $S_i\times D^N\times I$. Using the smoothing of the core on one flat side argument explained in great depth in this section, we make the tangential homeomorphism smooth on the core which is equivalent to $S_i\times D^{N}\times 0$. By conjugating by a smooth isotopy (after rounding corners) we can make the tangential homeomorphism smooth on $S_i\times (D^N\times 0\cup \d D^N\times I)$. Call this new bundle $W_i$ with base $\d_0W_i\cong M\times D^N\times 0$. At this point we can use by the linearization extension lemma \ref {extending linearizations from the spine} which implies that the linearization is unique up to homotopy and therefore can be ignored. This brings us to the lower left corner of the diagram.

The little cubes operation now produces the smooth bundle
\[
	M\times D^N\times I\cup_y \coprod W_i 
\]
Since $W_i$ has bottom and sides equal to $M\times (D^N\times 0\cup \d D^N\times I)$, we can lower $W_i$ into $M\times D^N\times I$ to make
\[
	M\times D^N\times I -\coprod M\times y_i(D^N)\times I\cup \coprod W_i
\]
This is a family of smooth bundles giving a homotopy of the mapping from the upper left to the lower right of our diagram. Since this new smooth structure on $M\times D^N\times I$ has support in the union of $S_i\times y_i(D^N)\times I$ and the $S_i$ are disjoint, we can expand the embeddings $y_i:D^N\to D^N$ until they are the identity and obtain an isotopy of the structure. The result is a smooth structure on $M\times D^N\times I$ given by $W_i$ on $S_i\times D^N\times I$ which is a description of the stabilization of $\sum M_i$. So, we have shown that the diagram commutes up to homotopy.
\end{proof}

%-----------------------------------------------------------------------------------
%            subsub section {Morlet's Theorem} A.5.6
%-----------------------------------------------------------------------------------

\subsubsection{Morlet's Theorem}

Since $D(M)$ can be chosen to have a trivial tangent bundle and can be stabilized by taking a limit with respect to all trivial disk bundles, we can now use the following theorem of Burghelea and Lashof which follows from Morlet's comparison theorem. (See \cite{BL74}, \cite[Thm H]{BL77}.)

\begin{thm} Let $X$ be a compact smooth manifold with trivial tangent bundle. Then the space of stable smooth structures on $X\times I$ equal to the standard smooth structure on $X\times 0\cup \d X\times I$ is a homology theory in $X$.
\end{thm}

Let $\cH^\%(X)$ denote this homology theory of $X$. This is the homology theory associated to the spectrum of $\cH(\ast)$:
\[
	\cH^\%(X)=\Omega^\infty(X_+\wedge \cH(\ast)).
\]
(See the section on homotopy theory below.) Then the theorem above together with the smoothing of the core theorem gives the following.

\begin{cor}\label{smoothing theorem elemental case} Let $X\times D^k\to D^k$ be the trivial bundle with fiber $X$. Then
\[
	\std{D^k}{S^{k-1}}(X\times D^k)\simeq \Omega^k\cH^\%(X)
\]% In particular, $\pi_0 \std{D^k}{S^{k-1}}(X\times D^k)$ is an abelian group.
\end{cor}

Taking $k=0$ we see that
\[
	\cH^\%(X)\simeq \stdone\ast(X)
\]
is the space of stable tangential smoothings of $X$. We can extend this calculation to the general case using the following lemma.

\begin{lem}[fibration lemma]
We have a fibration sequence:
\[
	\std{B}{\d_0B}(M)\to 
	\stdone{B}(M)\to 
	\stdone{\d_0B}(M)
\]
\end{lem}

\begin{proof}
If we choose a collar neighborhood $C$ of $\d_0B$ in $B$ and a topological product structure $M_C\cong M_{\d_0B}\times I$ we can easily extend deformations of smooth structures of $M$ over $\d_0B$ to deformations (with support in $M_C$) of the smooth structure of $M$.
\end{proof}

From the basic case given in Corollary \ref{smoothing theorem elemental case} and the fibration lemma the general case will easily follow:

\begin{thm}[main smoothing theorem]\label{main smoothing thm} Let $W\to B$ be a compact smooth manifold bundle. Then we have a natural homotopy equivalence
\[
	\g_W: \std{B}{\d_0B}(W)\simeq\Gamsub{B}{\d_0B}\cH^\%_B(W)
\]
where $\cH^\%_B(W)$ is the fiberwise $\cH^\%$ homology bundle of $W$ over $B$, i.e. the bundle whose fiber over $b\in B$ is $\cH^\%(W_b)$ where $W_b=p^{-1}(b)$.
\end{thm}

\begin{rem} Since the fibers of $\cH^\%_B(W)\to B$ are infinite loop spaces and $\cH(\ast)$ has finite type \cite{Dwyer80}, this implies that $\pi_0\std{B}{\d_0}(W)$ is a finitely generated abelian group.
\end{rem}

\begin{proof}
Choose a smooth triangulation of $B$ so that $\d_0B$ and $\d_1B$ are subcomplexes. Let $A$ be a subcomplex of $B$ containing the $k-1$ skeleton and let $\s$ be a $k$ simplex of $B$. Then we have a mapping between two fibration sequences:
\[
\xymatrix{
\std{\s}{\d\s}(W_\s)\ar[d]_\a\ar[r] &
	\stdone{A\cup\s}(W_{A\cup\s})\ar[d]_\b\ar[r] &
	\stdone{A}(W_A)\ar[d]_\g\\
\Gamsub{\s}{\d\s}\cH^\%_\s(W_\s) \ar[r]& 
	\Gam_{A\cup\s}\cH^\%_B(W) \ar[r]&
	\Gam_{A}\cH^\%_B(W)
	}%end xymatrix
\]
We use the excisiveness of smoothing to identify $\std{A\cup\s}{A}(W_{A\cup\s})=\std{\s}{\d\s}(W_\s)$. Since $\s$ is contractible, $W_\s$ is a product bundle $W_\s\cong \s\times W_b$. So, $\a$ is a homotopy equivalence by Corollary \ref{smoothing theorem elemental case}. If $\g$ is a homotopy equivalence then $\b$ will be a homotopy equivalence. Therefore, by induction on the number of simplicies, $\stdone{B}(W)\simeq\Gam_B\cH^\%_B(W)$. Another map of fibration sequences proves the relative version stated in the theorem.
\end{proof}

%%%%%%%%%%%%%%%%%%%%%%%%%%
%
%                subsubSection  {Stratified smoothing theorem} A.6.7.
%
%%%%%%%%%%%%%%%%%%%%%%%%%%

\subsubsection{Stratified smoothing theorem}\label{ss: stratified smoothing}

%We need the following naturality statements.

We will use the following trivial observation to extend the main smoothing theorem to the ``stratified'' case.

\begin{lem}[additivity of smoothing] Suppose that $E_i$ are disjoint smooth bundles over $B$. Then
\[
	\ttd{B,\d_0B}(\smallcoprod E_i, \smallcoprod \d\vv E_i)=\prod \ttd{B,\d_0B}(E_i,\d\vv E_i)
\]
\end{lem}

The basic case of the stratified smoothing theorem is the following. Suppose that $M$ is a smooth bundle over $B$ and $E$ is a codimension $0$ subbundle of $M$ which is a disjoint union of bundles: $E=\coprod E_i$. Then clearly,
\[
\ttd{B,\d_0}(E,\d\vv E)\cong \ttd{B,\d_0}(M,C)\subseteq \ttd{B,\d_0}(M,\d\vv M)
\]
where $C$ is the closure of the complement of $E$ in $M$.

By naturality of the homotopy equivalence $\g_E,\g_M$ we get the following.

\begin{lem}\label{gamma is compatible with inclusion}
After stabilization the homotopy equivalences $\g_E,\g_M$ given by Theorem \ref{main smoothing thm} are compatible with inclusion in the sense that the following diagram commutes.
\[
\xymatrix{
\std{B}{\d_0}(E)\ar[d]^\subseteq\ar[r]^(.45){\g_E}_(.45)\simeq &
	\Gamsub{B}{\d_0}\cH^\%_B(E)\ar[d]^\subseteq\\
\std{B}{\d_0}(M)\ar[r]^(.45){\g_M}_(.45)\simeq &
	\Gamsub{B}{\d_0}\cH^\%_B(M)
	}%end xymatrix
\]
\end{lem}

We need to extend this lemma to the case when $E$ is replaced with a stratified subbundle $E^\delta$ of $W$. In that case the space $\stdone B(E^\delta)$ and the bundle $\cH^\%_B(E^\delta)$ take new meanings which we now explain.

\begin{defn}
Suppose that $W$ is a smooth bundle over $B$ and we have a smooth triangulation of $B$ so that $\d_0B$ is a subcomplex. Over each simplex $\s$ of $B$, suppose we have a smooth codimension-0 compact manifold subbundle $E_\s\subseteq W_\s$. Since $\s$ is contractible, $E_\s\cong F\times\s$ for some compact manifold $F$. Suppose that $E_\t\subseteq E_\s$ for all $\t\subseteq\s$. Suppose also that $E_\t$ is empty for all $\t\subseteq\d_0B$. For example, we could let $E_\t=\emptyset$ if $\t\subseteq\d_0B$ and $E_\s=W_\s$ otherwise. (In general, $E_\t \subsetneq E_\s|\t$.) We let $E^\delta\subseteq W$ be the union over all simplices $\s$ of the restriction of $E_\s$ to the interior of $\s$. We say that $E^\delta$ is a \emph{stratified subbundle} of $W$. %The notation $()^\delta$ refers to the fact that $E$ is ``deleted'' over certain simplices in the sense that $E_\s$ is empty.
\end{defn}

Let $\stdone{B}(E^\delta)$ be the subspace of $\std{B}{\d_0}(W)$ of all tangential stable smoothings of $W$ which have support in the interior of $E_\s$ over the interior of each simplex $\s$. For example, if $B=\s$ is single simplex and $\d_0B=\d\s$ then $\stdone \s(E^\delta)=\std{\s}{\d\s}(E)$ since $E^\delta|\d\s$ is empty (``deleted'' from $E$). Since the complement of $E$ in $W$ is fixed, $\stdone B(E^\delta)$ is isomorphic to the space of stable tangential smoothings of $E$ which are fixed over any ``deleted'' simplex and on the vertical boundary $\d\vv E$ of $E$.

Let $\cH^\%_B(E^\delta)$ denote the stratified subbundle of $\cH^\%_B(W)$ which is equal to $\cH^\%_\s(E_\s)$ over the interior of every simplex $\s$.

\begin{thm}[stratified smoothing theorem]\label{stratified smoothing theorem} Suppose that $E^\delta$ is a stratified subbundle of $W$. Then we have a homotopy equivalence
\[
	\g_{E^\delta}:\stdone B(E^\delta)\simeq \Gam_B\cH^\%_B(E^\delta).
\]
with terms defined above.
\end{thm}

\begin{proof}
The argument is the same as in the main smoothing theorem (\ref{main smoothing thm}), but it is short and worth repeating. Let $A$ be a subcomplex of $B$ containing the $k-1$ skeleton and let $\s$ be a $k$-simplex. Then we have a mapping of fiber sequences:
\[
\xymatrix{
\std{\s}{\d\s}(E_\s)\ar[d]_\a\ar[r] &
	\stdone{A\cup\s}(E^\delta_{A\cup\s})\ar[d]_\b\ar[r] &
	\stdone{A}(E^\delta_A)\ar[d]_\g\\
\Gamsub{\s}{\d\s}\cH^\%_\s(E_\s) \ar[r]& 
	\Gam_{A\cup\s}\cH^\%_B(E^\delta) \ar[r]&
	\Gam_{A}\cH^\%_B(E^\delta)
	}%end xymatrix
\]
$\a$ is a homotopy equivalence by the main smoothing theorem and $\g$ is a homotopy equivalence by induction on the size of $A$. So $\b$ is a homotopy equivalence and the stratified smoothing theorem follows.
\end{proof}

We are interested in the following special case. Suppose that $L$ is a compact smooth $q$-manifold ($q=\dim B$) with $\d L=\d_0 L\cup \d_1 L$ where $\d_0L,\d_1L$ meet along a corner set $\d\d L$. Suppose that $\ll:L\to B$ is a smooth immersion with image disjoint from $\d_0B$ so that $\ll^{-1}(\d_1 B)=\d_1L$ and $\ll(\d_0L)$ meets $\d B$ transversely along $\ll(\d\d L)$. (See Figure \ref{fig: L to B example}.) Assume that the immersion $\ll$ is self-transverse, so that there exists a smooth triangulation of $B$ for which the number of inverse image points in $L$ is constant on each open simplex.
\begin{figure}[htbp]
\begin{center}
%
%\vs5
{
\setlength{\unitlength}{1in}
%\centerline
{\mbox{
\begin{picture}(3.5,1.2)
      \thicklines
%    \thinlines
{
      \put(0,1.1){\line(1,0){3}}
      \put(1.5,1.1){
      \qbezier(-1.3,0)(.5,-.8)(0,-.9)
      \qbezier(.9,0)(-.5,-.8)(0,-.9)
            \qbezier(-.2,0)(.9,-1)(0,-1.1)
            \qbezier(.4,0)(-.9,-1)(0,-1.1)
       }
       \put(3.1,1){$\d_1B$}
       }
       \put(1,.85){$\ll(L)$}
\end{picture}}
}}
%\vs5
\caption{In this example, $L$ is a square and $\ll:L\to B$ maps $\d_1L=$ two opposite sides into $\d_1B$. The diamond shaped region is covered twice by $\ll(L)$.}
\label{fig: L to B example}
\end{center}
\end{figure}

\begin{defn}\label{def:stratified embedding-immersion pair}
Let $\pi:E\to L$ be a compact manifold bundle with the same dimension as $W$. Then, by a \emph{stratified embedding-immersion pair} (or simpley \emph{stratified embedding}) we mean a smooth codimension 0 embedding $\tilde\ll:E\to W$ together with a codimension 0 immersion $\ll:L\to B$ as given above, making the following commuting diagram.
\[
%\xymatrixrowsep{10pt}\xymatrixcolsep{10pt}
\xymatrix{%begin xy matrix
E\ar[d]_\pi\ar[r]^{\tilde\ll}&
	W\ar[d]^p\\
L \ar[r]^\ll & 
	B
	}%end xy matrix
\]
\end{defn}

Given a stratified embedding-immersion pair $(\tilde\ll,\ll):E\to W$, it follows from the definition that $\tilde\ll^{-1}(W_{\d_1B})=E_{\d_1L}$. Also, the image in $W$ of the complement of $E_{\d_0L}$ in $E$ is a stratified subbundle of $W$ over $B$ with respect to a suitable triangulation of $B$. Call this image $E^\delta$. A fiberwise smooth structure for $E$ over $L$ which is equal to the given smooth structure over $\d_0L$ is equivalent to a fiberwise smooth structure on $W$ with support in $E^\delta$:
\[
	\std{L}{\d_0}(E)\cong \stdone{B}(E^\delta)\subseteq \std{B}{\d_0}(W)
\]
This implies that
\[
	\Gamsub{L}{\d_0}\cH^\%_L(E)\simeq
	\Gam_B\cH^\%_B(E^\delta)\subseteq 
	\Gamsub{B}{\d_0B}\cH^\%_B(W)
\]
%The statement that we need is the following.
So, we get the following key result which is the extension of Lemma \ref{gamma is compatible with inclusion} to the stratified case.

\begin{cor}\label{corollary of stratified smoothing theorem} If $(\tilde\ll,\ll):E\to W$ is a stratified embedding and $E^\delta\subset W$ is the corresponding stratified subbundle then the homotopy equivalences $\g_{E^\delta}$ of Theorem \ref{stratified smoothing theorem} is compatible with the homotopy equivalences $\g_E,\g_W$ of Theorem \ref{main smoothing thm}. I.e., the following diagram commutes.
\[
\xymatrix{
\std{L}{\d_0L} (E)\ar[d]_{\g_E}^\simeq\ar[r]^\mu_\simeq &
	\stdone{B} (E^\delta)\ar[d]_{\g_{E^\delta}}^\simeq\ar[r]^(.43)\subseteq&
	\std{B}{\d_0B} (W)\ar[d]_{\g_W}^\simeq\\
\Gamsub{L}{\d_0L}\cH^\%_L(E)\ar[r]^\mu_\simeq &
	\Gam_B\cH^\%_B(E^\delta)\ar[r]^(.43)\subseteq&
	\Gamsub{B}{\d_0B}\cH^\%_B(W)
	}%end xymatrix
\]
\end{cor}

We observe that both $\stdone{B} (E^\delta)$ and $\Gam_B\cH^\%_B(E^\delta)$ are independent of the decomposition $\d B=\d_0B\cup \d_1B$ as long as $\d_0B$ is disjoint from the image of $\ll:L\to B$.

 %%%%%%%%%%%%%%%%%%%%%%%%%%%%%%%%%%%%
 %
 %					End of Part A 
 %
 %%%%%%%%%%%%%%%%%%%%%%%%%%%%%%%%%%%%

 %\newpage

%%%%%%%%%%%%%%%%%%%%%%%%%%%%%%
%%%%%%%%%%%%%%%%%%%%%%%%%%%%%%
%
%		Appendix B: Homotopy theory (Computation of relative DWW-torsion}
%

%\setcounter{section}{10}% one less than the section number
%\setcounter{page}{100}% equal to the page number

\section{Homotopy theory}\label{Appendix B}

The computation of $\pi_0\Gamsub{B}{\d_0B}\cH^\%_B(M)\otimes\RR$ is an exercise in elementary homotopy theory which we will now explain. First we need to recall the definition of generalized homology.

 %%%%%%%%%%%%%%%%%%%%%%%%%%%%%%%%%%
 %
 %		subsection Review of Generalized homology
 %
 %%%%%%%%%%%%%%%%%%%%%%%%%%%%%%%%%%

\subsection{Review of generalized homology}\label{subsec21}
 
 We assume that all our spaces are Hausdorff and homotopy equivalent to CW-complexes. Suppose that $\bf G$ is a \emph{prespectrum}, i.e., a collection of pointed spaces $G_0, G_1, \cdots$ and pointed maps $\Sigma G_n\to G_{n+1}$ (which is equivalent to a pointed map $G_n\to \Omega G_{n+1}$). Then, for any pointed space $X$, we get another prespectrum $X\wedge \bf G$ with $n$-th space $X\wedge G_n$ since $\Sigma(X\wedge G_n)\cong X\wedge \Sigma G_n$. Two prespectra are considered to be the same if the spaces $G_n$ and structure maps $\Sig G_n\to G_{n+1}$ agree for sufficiently large $n$. Therefore, $G_n$ need only be defined for large $n$. We assume that $G_n$ is $n-1$ connected for large $n$.
 
 If $\xi$ is an $m$-dimensional vector bundle and $\e^k$ is the trivial $k$-plane bundle over the same base space then the Thom space $D(\xi\oplus\e^k)/S(\xi\oplus\e^k)$ of $\xi\oplus\e^k$ is the $k$-fold suspension of the Thom space $D(\xi)/S(\xi)$ of $\xi$. Define a prespectrum ${\bf T}(\xi)$ starting in degree $m$ so that $T(\xi)_{m+k}=D(\xi\oplus \e^{k})/S(\xi\oplus\e^k)$. This is the suspension spectrum of the formally desuspended usual Thom space:
\[
	T(\xi)=\Sig^{-m}D(\xi)/S(\xi)
\]
It is well-defined on the stable vector bundle associated to $\xi$ and $T_m(\xi)=D(\xi)/S(\xi)$ is $m-1$ connected. If $\xi$ is oriented then the Thom Isomorphism Theorem tells us that the reduced homology of $T(\xi)$ is isomorphic to the homology of the base space of $\xi$.

Associated to any prespectrum $\bf G$ we have the space
\[
	\Omega^\infty {\bf G}:=\colim \Omega^n G_n
\]
We will assume that the maps $G_n\to \Omega G_{n+1}$ are embeddings. Then $\Omega^\infty \bf G$ is an infinite loop space since $\Omega^\infty {\bf G}=\Omega(\Omega^\infty {\bf F})$ where $F_n=G_{n+1}$ is the connected \emph{delooping} of $\bf G$ which we denote ${\bf F}=\Omega^{-1}\bf G$.

The homology/reduced homology groups of a space $X$ with coefficients in the spectrum associated to $\bf G$ are defined to be the homotopy groups:
\[
	H_n(X;{\bf G}):=\pi_n(\Omega^\infty(X_+\wedge {\bf G}))=\colim \pi_{n+k}(X_+\wedge G_k)
\]
\[
	\ov H_n(X;{\bf G}):=\pi_n(\Omega^\infty(X\wedge {\bf G}))=\colim \pi_{n+k}(X\wedge G_k)
\]
where $X_+=X\coprod\ast$ is $X$ with an added disjoint basepoint. Section 8 of \cite{DWW} explains how any homotopy functor $\bf G$ gives an ``excisive'' functor ${\bf G}^\%(X)\simeq X_+\wedge {\bf G}(\ast)$ and when ${\bf G}(X)$ is a spectrum valued functor, meaning $G_n(X)\cong\Omega G_{n+1}(X)$, they used the same notation with a different meaning (note the font difference):
\[
G^\%(X):=\colim\Omega^n (G^\%)_n(X)\simeq \Omega^\infty(X_+\wedge {\bf G}(\ast))
\]
We will always use this second definition: $G^\%(X)=\Omega^\infty{\bf G}^\%(X)$. We also use the notation
\[
	\ov{G}^\%(X):=\Omega^\infty(X\wedge {\bf G}(\ast))
\]
so that $G^\%(X)=\ov{G}^\%(X_+)$. Then $\ov{G}^\%$ is a functor that takes cofibration sequences to fibration sequences and homotopy push-out squares to homotopy pull-back squares. In particular:
\begin{equation}\label{eq:ov(G)(X v Y)=ov(G)X x ov(G)Y}
\ov{G}^\%(X\vee Y)\simeq \ov{G}^\%(X)\times\ov{G}^\%(Y)
\end{equation}
\[
	\ov{G}^\%( X)\simeq\Omega\ov{G}^\%(\Sig X)
\]
\[
	G^\%(T(\xi))\simeq \Omega^{\dim\xi} G^\%(D(\xi)/S(\xi)).
\]
For any spectrum valued functor $\bf G$ so that $G_n$ is $n-1$ connected for all sufficiently large $n$, we call $G^\%$ a \emph{homology functor} and $\ov G^\%$ a \emph{reduced homology functor}.

 %%%%%%%%%%%%%%%%%%%%%%%%%%%%%%%%%%%%
 %
 %					subsection Fiberwise homology
 %
 %%%%%%%%%%%%%%%%%%%%%%%%%%%%%%%%%%%%
 %\newpage

\subsection{Fiberwise homology}\label{subsec22}

Suppose that $X\to M\xrarrow{p} B$ is a fiber bundle where $B$ is a compact oriented smooth $q$-manifold. Then let $G^\%_B(M)$ be the bundle over $B$ with fiber $G^\%(X)$. Since the fibers are pointed, this bundle has a trivial section.

If $A\subseteq B$ is a cofibration let $\Gamsub{B}{A} G^\%_B(M)$ be the space of sections of $G^\%_B(M)$ which are trivial on $A$. This is an infinite loop space since
\[
	\Gamsub{B}{A}G^\%_B(M)\cong \Omega^k\Gamsub{B}{A}(\Omega^{-k}G)^\%_B(M)
\]
where $\Omega^{-k}{G}=\{G_{k+n}\}$. In particular, $\pi_0\Gamsub{B}{A}G^\%_B(M)$ is an abelian group.

%%%%%%%%%%%%%%%%%%%%%%%%%%%%%%%%%%%% 
%
%	subsubsection (Statement of) {Theorem and corollary}
%
%%%%%%%%%%%%%%%%%%%%%%%%%%%%%%%%%%%%
 
\subsubsection{basic homotopy calculation}

Suppose that $\d B$ is a union of two $q-1$ dimensional submanifolds $\d_0 B,\d_1B$ which meet along their common boundary $\d\d_0B=\d\d_1B$. For any $A\subseteq B$ we use the notation $M_A=p^{-1}(A)$. Then

\begin{thm}[basic homotopy calculation]\label{thm: homotopy type of space of sections} For any homology functor $G^\%$, there is a natural homotopy equivalence
\[
\f_B:\Gamsub{B}{\d_0 B}G^\%_B(M) \simeq \Omega^{q}\ov{G}^\%(T(\xi)/T(\xi_1))
\]
where $T(\xi)$ is the Thom space of $\xi$, the pull-back to $M$ of the stable normal bundle of $B$, and $T(\xi_1)\subseteq T(\xi)$ is the subspace given by restricting $\xi$ to $M_{\d_1B}$.
\end{thm}

The proof of this fact is very similar to the proof of Poincar\'{e} duality and is explained below. Here is the example that we have in mind.

\begin{cor}\label{main homotopy calculation}
Let $\cH(X)$ be the space of stable $h$-cobordisms of $X$. Then there exists a natural isomorphism
\[
	\th_M:\pi_0\Gamsub{B}{\d_0 B}\cH^\%_B(M)\otimes\RR\cong \bigoplus_{k>0}H_{q-4k}(M,M_{\d_1B};\RR)
\]
where $q=\dim B$.
\end{cor}

\begin{proof}
By the theorem we have
\[
	\pi_0\Gamsub{B}{\d_0 B}\cH^\%_B(M)\cong \ov{H}_{q}(T(\xi)/T(\xi_1);\cH(\ast)).
\]
But we have a rational equivalence of infinite loop spaces \cite{BW87}:
\[
	\cH(\ast)\simeq_\QQ G/O\simeq_\QQ \prod_{k>0} K(\ZZ,4k)
\]
So, rationally we have:
\[
	\ov{H}_{q}(T(\xi)/T(\xi_1);\cH(\ast))\cong_\QQ \bigoplus_{k>0}\ov{H}_{q-4k}(T(\xi)/T(\xi_1))\cong_\QQ \bigoplus_{k>0}H_{q-4k}(M,M_{\d_1B})
\]
using the Thom isomorphism theorem at the last step. Extend scalars to $\RR$ to get the result. The specific choice $\th_M$ of this natural isomorphism is given below.
\end{proof}

%-----------------------------------------------------------------------------------
%            subsub section {definition of $\theta,\Theta$}
%-----------------------------------------------------------------------------------

\subsubsection{normalization of $\theta_M,\Theta_M$}\label{subsection222}

To make a specific choice for the isomorphism in Corollary \ref {main homotopy calculation} we need a specific rational homotopy equivalence
\[
	\cH(\ast)\simeq_\QQ\prod_{k>0}K(\ZZ,4k)
\]
This is equivalent to choosing a fixed rational cohomology class in $\prod_{k>0}H^{4k}(\cH(\ast);\QQ)$. We take this to be the higher \IK-torsion invariant which is a real cohomology class
\[
	\t^\IK\in \prod_{k>0}H^{4k}(\cH(\ast);\RR)
\]
which comes from an element of $\prod_{k>0}H^{4k}(\cH(\ast);\zeta(2k+1)\QQ)$ by \cite{I:Axioms0}.

With this choice we get a natural isomorphism
\[
	\theta_M:\pi_0\Gamsub{B}{\d_0 B}\cH^\%_B(M)\otimes\RR\xrarrow{\cong}  \bigoplus_{k>0}H_{q-4k}(M,M_{\d_1B};\RR)
\]
%Combining this with the homotopy equivalence $\std{B}{\d_0}(M)\simeq \Gamsub{B}{\d_0 B}\cH^\%_B(M)$ given by Theorem \ref {main smoothing thm} we get the following.

\begin{thm}\label{computation of stable smooth structures}
We have a natural isomorphism
\[
	\Theta_M:\pi_0\std{B}{\d_0B}(M)\otimes\RR\cong  \bigoplus_{k>0}H_{q-4k}(M,M_{\d_1B};\RR).
\]
\end{thm}

\begin{proof} 
Take $\Theta_M=\th_M\circ (\g_M)_\ast$ where $(\g_M)_\ast=\pi_0(\g_M)\otimes \RR$ is the isomorphism induced by the homotopy equivalence
\[
	\g_M:\std{B}{\d_0B}(M)\simeq \Gamsub{B}{\d_0 B}\cH^\%_B(M)
\]
of Theorem \ref {main smoothing thm}.
\end{proof}

In the case when $M\to B$ is a linear disk bundle, the elements of $\std{B}{\d_0B}(M)$ are represented by disk bundles $M'\to B$ which are linear over $\d_0B$. The higher \IK-torsion invariant $\t^\IK(M')\in \bigoplus_{k>0}H^{4k}(B,\d_0B;\RR)$ is zero if $M'$ is a linear bundle over all of $B$. The choice of the isomorphisms $\th_M,\Theta_M$ is designed to make the following equation hold:

\begin{prop}\label{prop: normalization of theta} %B.2.3
Suppose that $M\to B$ is a linear disk bundle and $M'\to B$ represents an element of $\pi_0\std{B}{\d_0B}(M)$. Then
\[
	D\t^\IK(M')=p_\ast\Theta_M(M')
\]
where $D$ is Poicar\'e duality and $p_\ast$ is the map in homology induced by $p:M\to B$. In other words, the following diagram commutes.
\[
%\xymatrixrowsep{10pt}\xymatrixcolsep{10pt}
\xymatrix{%begin xy matrix
\pi_0\std{B}{\d_0B}(M)\otimes\RR\ar[d]^\cong_{\Theta_M}\ar[r]^{\t^\IK}&
	\bigoplus_{k>0}H^{4k}(B,{\d_0B};\RR)\ar[d]_{\cong}^{D} \\
	\bigoplus_{k>0}H_{q-4k}(M,M_{\d_1B};\RR)
	\ar[r]_{\cong}^{p_\ast}&
	\bigoplus_{k>0}H_{q-4k}(B,{\d_1B};\RR)
	}%end xy matrix
\]
\end{prop}

\begin{proof}
This follows from the definition of $\Theta_M$ and the fact that $\t^\IK(M)=0$ for linear disk bundles making $\t^\IK(M',M)=\t^\IK(M')-\t^\IK(M)=\t^\IK(M')$.
\end{proof}

\begin{rem}
The immediate corollary that $\t^\IK$ is an isomorphism in the above diagram is reformulation of the result of Farrell-Hsiang \cite{FarrellHsiangDiffDn} using the properties of higher torsion (see, e.g., \cite{I:Axioms0}).
\end{rem}

%%%%%%%%%%%%%%%%%%%%%%%%%%%%%%%%%%%%%
%
%					subsection {Proof of the theorem}
%
%%%%%%%%%%%%%%%%%%%%%%%%%%%%%%%%%%%%%
 
\subsection{Proof of Theorem \ref{thm: homotopy type of space of sections}}\label{subsec23}

 We prove the theorem first in the special case when $B$ is a compact $q$-manifold embedded in $D^q$ and $\d_0B$ is empty (so $\d_1B=\d B$). In that case the normal bundle of $B$ is trivial, so $T(\xi)=M_+$ and $T(\xi)/T(\xi_1)=M/M_{\d B}$. Let
 \[
 	\f_B:\Gam_BG^\%_B(M)\to\Omega^q\ov{G}^\%(M/M_{\d B})
 \]
 be the map given as follows.
 Take the inclusion of $M$ into the trivial bundle $B\times M$ via the map $(p,id_M):M\to B\times M$. This induces a map
\[
	\psi_B:\Gam_BG_B^\%(M)\to \Gam_BG_B^\%(B\times M)=Map(B,G^\%(M)).
\]
For any $\g\in \Gam_BG_B^\%(M)$ the mapping $\psi_B(\g):B\to G^\%(M)$ sends $\d B$ into $G^\%(M_{\d B})$. So, it induces a mapping
\[
	\f_B(\g):D^q/S^{q-1}\to D^q/(D^q\backslash int B)=B/\d B\xrarrow{\psi_B(\g)}  \ov{G}^\%(M/M_{\d B})
\]
representing an element of $\Omega^q\ov{G}^\%(M/M_{\d B})$. In the relative case, $\psi_{(B,\d_0B)}(\g)$ sends $\d_1B$ into $G^\%(M_{\d_1B})$ and $\d_0B$ into $\ast=G^\%(\emptyset)\subset G^\%(M_{\d_1B})$. So $\psi_{(B,\d_0B)}(\g)$ induces a mapping
\[
	\f_{(B,\d_0B)}(\g):D^q/S^{q-1}\to B/\d B\to  \ov{G}^\%(M/M_{\d_1 B})
\]
giving an element of $\Omega^q\ov{G}^\%(M/M_{\d_1 B})$.

\begin{lem}\label{sections lemma}
Suppose that $B$ is a compact $q$-manifold embedded in $D^q$. Then the mapping
\[
	\f_B:\Gam_BG^\%_B(M)\to\Omega^q\ov{G}^\%(M/M_{\d B})
\]
described above is a homotopy equivalence.
\end{lem}

Suppose for a moment that this is true. 

\begin{proof}[Proof of Theorem \ref{thm: homotopy type of space of sections}] 
Consider the next case when $B$ is a compact $q$-manifold embedded in the interior of $D^q$ and $\d B=\d_0B\cup\d_1B$. Let $C\cong \d_0B\times I$ be an external collar neighborhood for $\d_0B$ in $D^q$ so that $J=B\cup C\simeq B$ and $B\cap C=\d_0B$. Then the bundle $M$ over $B$ extends to a bundle $M_J\to J$ which is unique up to isomorphism and the mappings $\f_J,\f_C,\f_{(B,\d_0B)}$ are compatible making the following diagram commute.
\[
\xy
\xymatrix{
{\Gamsub{B}{\d_0 B} G^\%_B(M)}\ar[d]^{\f_{(B,\d B)}}\ar[r]
&
 {\Gam_{J} G^\%_{J}(M_J)}\ar[d]^{\f_J}\ar[r]
 &{\Gam_{C} G^\%_C(M_C)}\ar[d]^{\f_C}
 \\
{\Omega^q\ov G^\%(M/M_{\d_1B})} \ar[r] 
&
 {\Omega^q\ov{G}^\%(M_J/M_{\d J})}\ar[r] 
 &
{\Omega^q\ov{G}^\%(M_C/M_{\d C})}
}
\endxy
\]
The top row is a fibration sequence since $\Gamsub{B}{\d_0B}G^\%_B(M)=\Gamsub{J}{C}G^\%_J(M_J)$ is the fiber of the restriction map ${\Gam_{J} G^\%_{J}(M_J)}\to {\Gam_{C} G^\%_C(M_C)}$ and the bottom row is a fibration sequence since
\[
	M/M_{\d_1B}\to M_J/M_{\d J} \to M_C/M_{\d C}
\]
is a cofibration sequence. Since $\f_J,\f_C$ are homotopy equivalences by the lemma above, the induced map $\f_{(B,\d_0B)}$ is also a homotopy equivalence and $T(\xi)/T(\xi_1)=M/M_{\d_1B}$, so the theorem holds in this case.

In the general case we choose an embedding $B^q\hookrightarrow D^{q+n}$ and let $\nu$ be the $n$-dimensional normal bundle of $B$. Let $\xi$ be the pull back of $\nu$ to $M$ and let $D(\nu),S(\nu),D(\xi),S(\xi)$ be the corresponding disk and sphere bundles. Then $D(\xi),S(\xi)$ are the pull-backs of $D(\nu),S(\nu)$ to $M$ and therefore, $D(\xi)\to D(\nu)$ is a fibration with fiber $X$ and $S(\xi)=D(\xi)_{S(\nu)}$. Since $\d_1D(\nu)=S(\nu)\cup D(\nu)_{\d_1B}$, we have $D(\xi)_{\d_1D(\nu)}=S(\xi)\cup D(\xi_1)$. So,
\[
	 D(\xi)/D(\xi)_{\d_1D(\nu)}=\Sig^n(T(\xi)/T(\xi_1))
\]
making $\b$ a homotopy equivalence in the diagram below. Since $(B,\d_0 B,M)\simeq (D(\nu),D(\nu)_{\d_0 B},D(\xi))$, $\a$ is also a homotopy equivalence. 

Finally, $D(\nu)$ is a $q+n$ submanifold of $D^{q+n}$ and the closure of the complement of $\d_0D(\nu)=D(\nu)_{\d_0B}$ in $\d D(\nu)$ is $\d_1D(\nu)$. So, the map $\f= {\f_{(D(\nu),D(\nu)_{\d_0 B})}}$ is a homotopy equivalence by Lemma \ref{sections lemma} which we are assuming. Therefore, the composition is a homotopy equivalence, proving the theorem:
\[
\xymatrix{
\Gamsub{B}{\d_0 B}G^\%_B(M) \ar[d]_\a \ar[r]^(.45)\simeq&
\Omega^{q}\ov G^\%(T(\xi)/T(\xi_1)) \\
\Gamsub{D(\nu)}{D(\nu)_{\d_0 B}}G^\%_{D(\nu)}(D(\xi)) \ar[r]^\f & 
\Omega^{q+n}\ov{G}^\%(D(\xi)/D(\xi)_{\d_1D(\nu)})\ar[u]^\b
}
\]
\end{proof}

\begin{proof}[Proof of Lemma \ref{sections lemma}]
Suppose first that $B=D^q$. Then
\[
	\Gam_BG^\%_B(M)\cong Map(D^q,G^\%(X))\simeq G^\%(X)\simeq \Omega^q\ov{G}^\%(\Sig^q(X_+))\cong\Omega^q\ov{G}^\%(M/M_{\d B})
\]
and this homotopy equivalence is given by $\f_B$. 

In general we can choose a finite covering of $B$ by closed $q$-disks $A_i$ which is a ``good covering'' in the sense that the intersection of any finite number of $A_i$ is either empty or homeomorphic to a $q$-disk. Let $C=A_1\cup \cdots\cup A_{k-1}$ and $B=C\cup A_k$. By induction on $k$ we know that $\f_C,\f_{C\cap A_k}$ and $\f_{A_k}$ are homotopy equivalences. Now look at the commuting cubical diagram given by mapping each object of the left hand square to the corresponding object of the right hand square in the following diagrams.
\[
\xy
\xymatrix{
{\Gam_BG^\%_B(M)}\ar[r]\ar[d]
&{\Gam_{A_k}G^\%_{A_k}(M_{A_k})}\ar[d] 
&
{\Omega^q\ov{G}^\%(M|{ B})}\ar[r]\ar[d]
&{\Omega^q\ov{G}^\%(M|{A_k})}\ar[d]
\\
{\Gam_CG^\%_C(M_C)}\ar[r]
& {\Gam_{A_k\cap C}G^\%_{A_k\cap C}(M_{A_k\cap C})} 
& 
{\Omega^q\ov{G}^\%(M|C)}\ar[r] 
&{\Omega^q\ov{G}^\%(M|{A_k\cap C})}
}
\endxy
\]
Here $M|C=M_C/M_{\d C}\cong M/M_{B\backslash int C}$ and similarly for $C$ replaced with $B,A_k,A_k\cap C$. Since the functors $X\mapsto\Gam_X G_X^\%(M_X)$ and $\Omega^n\ov{G}^\%$ send cofiber squares to fiber squares, both squares are fiber squares. This implies that $\f_B:\Gam_BG^\%_B(M)\to \Omega^q\ov{G}^\%(M|{ B})$ is a homotopy equivalence as claimed.
\end{proof}

% subsection

%\newpage
%-----------------------------------------------------------------------------------
%            sub section {Stratified subbundles}
%-----------------------------------------------------------------------------------

\subsection{Stratified subbundles}\label{subsec24}

Recall from Definition \ref{def:stratified embedding-immersion pair} that a stratified embedding-immersion pair consists of a codimension $0$ immersion $\ll:(L,\d_1L)\to (B,\d_1B)$ covered by an embedding of smooth bundles $\tilde\ll:E\to W$ of the same dimension. By the Corollary \ref{corollary of stratified smoothing theorem} of the stratified smoothing Theorem \ref{stratified smoothing theorem} we have the following commuting diagram where the maps $\mu,\g$ are homotopy equivalences.
\[
\xymatrix{
\std{L}{\d_0L} (E)\ar[d]^\simeq_{\g_E}\ar[r]^\mu_\simeq &
	\stdone{B} (E^\delta)\ar[d]^\simeq_{\g_{E^\delta}}\ar[r]^(.43)\subseteq&
	\std{B}{\d_0B} (W)\ar[d]^\simeq_{\g_W}\\
\Gamsub{L}{\d_0L}\cH^\%_L(E)\ar[r]^\mu_\simeq &
	\Gam_B\cH^\%_B(E^\delta)\ar[r]^(.43)\subseteq&
	\Gamsub{B}{\d_0B}\cH^\%_B(W)
	}%end xymatrix
\]

We need to prove that our calculation of the bottom three terms is compatible with the two arrows.

\begin{thm}[stratified homotopy calculation]\label{stratified homotopy calculation} The following diagram commutes for any homology theory $G^\%$.
\[
\xymatrix{
\Gamsub{L}{\d_0L}G^\%_L(E)\ar[d]^\simeq_{\f_L}\ar[r]^\mu_\simeq &
	\Gam_BG^\%_B(E^\delta)\ar[r]^(.43)\subseteq&
	\Gamsub{B}{\d_0B}G^\%_B(W)\ar[d]^\simeq_{\f_B}\\
\Omega^q\ov G^\%(T(\xi_E)/T(\xi_{\d_1E})) \ar[rr]^{\tilde\ll_\ast}& 
	&
	\Omega^q\ov G^\%(T(\xi)/T(\xi_{\d_1W}))
	}%end xymatrix
\]
Here $\xi=p^\ast\nu_B$ is the pull-back of the normal bundle $\nu_B$ of $B$ to $W$ and $\xi_E=\tilde\ll^\ast\xi$. The bottom arrow is induced by the inclusion $T(\xi_E)\into T(\xi)$ given by $\tilde\ll:E\to W$. The mapping $\mu$ is the natural homotopy equivalence described below.
\end{thm}

The proof is given below, following two corollaries and two lemmas.

Since $\ll:L\to B$ is a codimension 0 immersion, the normal bundle of $B$ pulls back to the normal bundle of $L$: $\nu_L=\ll^\ast\nu_B$. Since $p\circ \tilde\ll=\ll\circ\pi:E\to B$, it follows that $\xi_E=\tilde\ll^\ast\xi\cong \pi^\ast \nu_L$. So, both vertical arrows in the diagram are the homotopy equivalences of the previous theorem.

The mapping $\mu$ can be described as follows. For any $b\in B$ let $x_1,\cdots,x_k$ be the elements of $L-\d_0L$ which map to $b$. Then $E^\delta_b=\coprod \tilde\ll(E_{x_i})$. So
\[
	G^\%(E_b^\delta)\simeq \prod G^\%(\tilde\ll E_{x_i})
\]
where the projection map $(E_b^\delta)_+\to (\tilde\ll E_{x_i})_+$ is the identity on $\tilde\ll E_{x_i}$ and sends the other components to the disjoint base point. (Then apply $\ov G^\%(X_+)=G^\%(X)$.)

There is a sixth space which can be inserted in the middle of the bottom arrow of the above diagram: $\Omega^q\ov G^\%(T(\xi_{E^\delta})/T(\xi_{\d_1E^\delta}))$ where $\xi_{E^\delta}$ is the restriction of the bundle $\xi$ to $E^\delta$ and $\xi_{\d_1E^\delta}$ is the restriction of $\xi$ to $\d_1E^\delta=E^\delta\cap W_{\d_1B}$.

The case that interests us is $G=\cH$ where, using the Thom Isomorphism Theorem we have the following.

\begin{cor}\label{cor of stratified homotopy}
The following diagram commutes where both horizontal arrows are induced by the stratified embedding $\tilde\ll:E\to W$.
\[
\xymatrix{
\pi_0\Gamsub{L}{\d_0L}\cH^\%_L(E)\otimes \RR\ar[d]^\cong_{\theta_E}\ar[rr]^{\tilde\ll_\ast} &
	&
	\pi_0\Gamsub{B}{\d_0B}\cH^\%_B(W)\otimes\RR\ar[d]^\cong_{\theta_W}\\
\bigoplus_{k>0}H_{q-4k}(E,\d_1E;\RR) \ar[rr]^ {\tilde\ll_\ast}& 
	&
	\bigoplus_{k>0}H_{q-4k}(W,\d_1W;\RR)
	}%end xymatrix
\]
\end{cor}

Combining this with Corollary \ref{corollary of stratified smoothing theorem} we obtain the following.

\begin{cor}\label{main corollary}
The isomorphism $\Theta_W=\th_W\circ(\g_W)_\ast$ is compatible with stratified embeddings $\tilde\ll:E\to W$ in the sense that the following diagram commutes.
\[
\xymatrix{
\pi_0\std{L}{\d_0L} (E)\otimes \RR\ar[d]^\cong_{\Theta_E}\ar[rr]^{\tilde\ll_\ast} &
	&
	\pi_0\std{B}{\d_0B} (W)\otimes\RR\ar[d]^\cong_{\Theta_W}\\
\bigoplus_{k>0}H_{q-4k}(E,\d_1E;\RR) \ar[rr]^ {\tilde\ll_\ast}& 
	&
	\bigoplus_{k>0}H_{q-4k}(W,\d_1W;\RR)
	}%end xymatrix
\]
\end{cor}

\begin{lem}
There is a homotopy equivalence $\mu:\Gamsub{L}{\d_0L}G^\%_L(E)\to \Gam_BG^\%_B(E^\delta)$ which sends $\g$ to the section $\mu(\g)$ which sends $b\in B$ to $(\g(x_i))_i\in \prod G^\%(\tilde\ll E_{x_i})$.
\end{lem}

As before, the proof of the theorem relies on the following lemma which does the case when $B^q$ is embedded in $D^q$.

\begin{lem}\label{second lemma for stratified homotopy calculation}
Suppose that $B^q$ is embedded in the $q$-disk $D^q$. Then
\[
	\Gam_B G^\%_B(E^\delta)\simeq \Gam_{L,\d_0L}G^\%_L(E)\simeq \Omega^q \ov G^\%(E^\delta/E^\delta_{\d B})\simeq \Omega^q \ov G^\%(E/E_{\d_1L})
\]
Furthermore the mapping $\f_B:\Gam_B G^\%_B(E^\delta)\to \Omega^q \ov G^\%(E^\delta/E^\delta_{\d B})$ giving this homotopy equivalence is natural with respect to restriction and inclusion as explained below.
\end{lem}

\emph{Naturality with respect to inclusion} means that the following diagram commutes for any choice of decompositions $\d B=\d_0B\cup\d_1B$ so that $E^\delta_{\d_0B}$ is empty. The horizonal maps in the diagram are induced by the inclusion $E^\delta\into W$.
\[
\xymatrix{
\Gam_B G^\%_B(E^\delta)\ar[r]\ar[d]_{\f_B} &
\Gam_{B,\d_0B} G^\%_B(W) \ar[d]^{\f_{B,\d_0}}\\
	\Omega^q \ov G^\%(E^\delta/E^\delta_{\d B})\ar[r]&
	\Omega^q \ov G^\%(W/W_{\d_1 B})
	}%end xymatrix
\]

\emph{Naturality with respect to restriction} means that the following diagram commutes assuming that $A\subseteq B$ is a $q$-submanifold transverse to the image of $\ll: L\to B$. 
\[
\xymatrix{
\Gam_B G^\%_B(E^\delta)\ar[r]\ar[d]_{\f_B} &
\Gam_{A} G^\%_A(E^\delta_A) \ar[d]^{\f_{A}}\\
	\Omega^q \ov G^\%(E^\delta/E^\delta_{\d B})\ar[r]&
	\Omega^q \ov G^\%(E^\delta_A/E^\delta_{\d A})
	}%end xymatrix
\]
The top horizontal arrow is given by restriction of sections to $A$ and the bottom horizonal arrow is induced by the quotient map $E^\delta/E^\delta_{\d B}\to E^\delta_A/E^\delta_{\d A}$.

\begin{proof}[Proof of Lemma \ref{second lemma for stratified homotopy calculation}]
The proof is basically the same as the proof of Lemma \ref{sections lemma}. First we consider the elemental case in which $B=D^q$ and $L$ is a disjoint union of $q$-disks $L_i$ with embeddings $\ll_i:L_i\to B$ so that $\ll_i^{-1}(\d B)=\d_1L_i$. Let $E_i\subseteq W$ be the image of $E_{L_i}$. Then each $E_i$ falls into one of three elemental cases:
\begin{enumerate}
\item[(0)] $\d_1L_i$ is empty. Then 
\[
\Gam_B G_B^\%(E_i)\cong \Gam_{L_i,\d L_i} G^\%_{L_i}(E)\cong Map(L_i/\d L_i,G^\%(E))\simeq\Omega^q G^\%(E_i)=\Omega^q \ov G^\%(E_{i\,+})
\]
\item $\d_1L_i$ and $\d_0L_i$ are $q-1$ disks meeting along a $q-2$ sphere. In this case,
\[
	\Gam_B G_B^\%(E_i)\simeq \ast\simeq \Omega^q \ov G^\%(E_i/E_{i\d_1L_i})
\]
\item $\d_1 L_i=\d B=S^{q-1}$ and $\d_0L_i$ is empty. Then $L_i=B$ and
\[
	\Gam_B G_B^\%(E_i)\cong  Map(B,G^\%(E_i))\simeq G^\%(E_i)\simeq \Omega^q \ov G^\%(E_i/E_{i\,\d B})
\]
\end{enumerate}
Therefore,
\[
	\Gam_B G_B^\%(E_i)\simeq \Omega^q\ov G^\%(E_i/E_{i,\d B})
\]
for each $i$ and we conclude that
\[
\Gam_B G_B^\%(E^\delta)\simeq\prod \Gam_B G_B^\%(E_i)\simeq
\prod\Omega^q\ov G^\%(E_i/E_{i,\d B})\simeq
\Omega^q\ov G^\%(E^\delta/E^\delta_{\d B}).
\]

In general we can choose a finite covering of $B$ by closed $q$-disks $A_i$ which is a ``good covering'' in the sense that the intersection of any finite number of $A_i$ is either empty or homeomorphic to an $q$-disk and the restriction of $E$ to each of these disks is elemental as described above. It is easy to do this very explicitly. First subdivide once to make sure the triangulation is sufficiently fine. Choose any fixed positive $\e<1/{q+1}$. For every simplex $\s$ take the set $N(\s)$ of all points $b\in B$ so that $t_i\le\e$ for every barycentric coordinate $t_i$ of $b$ corresponding to a vertex $v_i$ not in $\s$. Then $N(\s)$ is a polyhedron, being given by linear inequalities of barycentric coordinates and it is the closure of its interior which contains $\s$ as a deformation retract and is thus contractible. Therefore $N(\s)$ is a $q$-disk. Also, it is obvious that $N(\s)\cap N(\t)=N(\s\cap \t)$. Also, similar arguments show that each component of $L$ over $N(\s)$ is a $q$ disk and falls into one of the three cases discussed above. So $A_i=N(\s_i)$ form a good covering. 

The rest of the proof is almost word-for-word the same as the second half of the proof of Lemma \ref{sections lemma} except that we need $E$ to be replaced with $E^\delta$ and we need two more commuting squares with $B$ replaced by $L$ and $C$ and $A_k$ replaced by their inverse images in $L$. Then we have four fiber squares in which corresponding terms are homotopy equivalent by induction on $k$ proving the first part of the lemma.

It remains to show that the mapping 
\[
	\f_B:\Gam_B G^\%_B(E^\delta)\to \Omega^q \ov G^\%(E^\delta/E^\delta_{\d B})
\]
which gives the homotopy equivalence is natural with respect to inclusion and restriction. But this follows from the definitions. A section $\g$ of $G^\%_B(E^\delta)$ sends a point $b\in B$ to $\g(b)\in G^\%(E_b)\subseteq G^\%_B(E^\delta)$. The corresponding map $\f_B(\g):D^q/S^{q-1}\to \ov G^\%(E/E_{\d B})$ sends $b\in B\subseteq D^q$ to $\g(b)\in G^\%(E_b)\to \ov G^\%(E/E_{\d B})$. This is clearly compatible with inclusion: we simply map these images into larger sets. This is also compatible with restriction: the points $b\in A$ are sent to the same points as before and $b\notin A$ are sent to the basepoint by both $\g$ and $\f_A(\g)$.
\end{proof}

\begin{proof}[Proof of stratified homotopy calculation \ref{stratified homotopy calculation}]

The proof is analogous to the second paragraph in the proof of Theorem \ref{thm: homotopy type of space of sections}. Choose an embedding $B^q\hookrightarrow D^{q+n}$, take $\nu$ to be the $n$-dimensional normal bundle of $B$ and let $\xi,\nu_L,\xi_E$ be the pull backs of $\nu$ to $W,L,E$ respectively. Let $D(\nu),S(\nu),D(\xi)$, etc. be the corresponding disk and sphere bundles. Then we have a new stratified embedding-immersion pair:
\[
%\xymatrixrowsep{10pt}\xymatrixcolsep{10pt}
\xymatrix{%begin xy matrix
	D(\xi_E)\ar[d]\ar[r]^{\tilde\ll'} &
	D(\xi)\ar[d]\\
	D(\nu_L) \ar[r]^{\ll'}&
	D(\nu)
	}%end xy matrix
\]
giving a stratified subbundle $D(\xi_E)^\delta\subseteq D(\xi)$. Also, let

$\d_0D(\nu)=D(\nu)_{\d_0B}\simeq \d_0B$

  $\d_0D(\nu_L)=D(\nu_L)_{\d_0L}\simeq \d_0L$.
  
\noi Then:

$\d_1D(\nu)=D(\nu)_{\d_1B}\cup S(\nu)$

$D(\xi)_{\d_1D(\nu)}=D(\xi_{\d_1W})\cup S(\xi)$

$\d_1D(\nu_L)=(\ll')^{-1}\d D(\nu)=D(\nu_L)_{\d_1L}\cup S(\nu_L)$
  
 $D(\xi_E)_{\d_1D(\nu_L)}=D(\xi_{\d_1E})\cup S(\xi_E)$.

Since $D(\nu)$ is $q+n$ manifold in $D^{q+n}$, Lemma \ref{second lemma for stratified homotopy calculation} applies and we get the following commuting diagram where the vertical arrows are homotopy equivalences.
\[
\xymatrix{
\Gamsub{D(\nu_L)}{\d_0}G_{D(\nu_L)}^\%(D(\xi_E))\ar[d]^\simeq_{\f_{D(\nu_L)}}\ar[r]^\mu_\simeq &
	\Gam_{D(\nu)}G^\%_{D(\nu)}(D(\xi_E)^\delta)\ar[r]^\subseteq\ar[d]^\simeq_{\f_{D(\nu)}}&
	\Gamsub{D(\nu)}{\d_0}G^\%_{D(\nu)}({D(\xi)})\ar[d]^\simeq_{\f_{D(\nu),\d_0}}\\
\Omega^{q+n}\ov G^\%\left(\frac{D(\xi_E)}{D(\xi_E)_{\d_1\!D(\nu_L)}}\right) \ar[r]^{\mu'}_{\simeq}& 
	\Omega^{q+n}\ov G^\%\left(\frac{D(\xi_E)^\delta}{D(\xi_E)^\delta_{\d D(\nu)}}\right)\ar[r]^{\tilde\ll_\ast}&
	\Omega^{q+n}\ov G^\%\left(\frac{D(\xi)}{D(\xi)_{\d_1\!D(\nu)}}\right)
	}%end xymatrix
\]

Each term in the diagram above is homotopy equivalent to the corresponding term in the following diagram, proving the theorem.

\[
\xymatrix{
\Gamsub{L}{\d_0L}G^\%_L(E)\ar[d]^\simeq_{\f_L}\ar[r]^\mu_\simeq &
	\Gam_BG^\%_B(E^\delta)\ar[r]^(.43)\subseteq\ar[d]^\simeq_{\f_B}&
	\Gamsub{B}{\d_0B}G^\%_B(W)\ar[d]^\simeq_{\f_{B,\d_0\!B}}\\
\Omega^q\ov G^\%\left(\frac{T(\xi_E)}{T(\xi_{\d_1\!E})}\right) \ar[r]^{\mu'}_{\simeq}& 
	\Omega^q\ov G^\%\left(\frac{T(\xi_{E^\delta})}{T(\xi_{\d_1\!E^\delta})}\right)\ar[r]^{\tilde\ll_\ast}&
	\Omega^q\ov G^\%\left(\frac{T(\xi)}{T(\xi_{\d_1\!W})}\right)
	}%end xymatrix
\]

\end{proof}

This completes the proof of Theorem \ref{stratified homotopy calculation} which implies Corollary \ref{main corollary} that the rational calculation of the space of stable smooth structures given in Theorem \ref{computation of stable smooth structures} is compatible with stratified embeddings. This is used in the companion paper \cite{First} to transfer our more or less complete understanding of the rational stable exotic smooth structures on disk bundles and their relation to higher Reidemeister torsion to a corresponding understanding of rational stable exotic smooth structures on smooth manifold bundles with odd dimensional fibers.

 %%%%%%%%

 %%%%%%%%%%%%%%%%%%%%%%%%%%%%%%%%%%%%
 %
 %					bibliography 
 %
 %%%%%%%%%%%%%%%%%%%%%%%%%%%%%%%%%%%%

\bibliographystyle{amsplain}

 %%%%%%%%%%%%%%%%%%%%%%%%%%%%%%%%%%%%
 %
 %					end of the document 
 %
 %%%%%%%%%%%%%%%%%%%%%%%%%%%%%%%%%%%%

\end{document}